\documentclass[12pt]{article}

\usepackage{a4wide}
\usepackage{fancyhdr}
\usepackage{graphicx}
\usepackage{amsfonts,amsmath,amssymb,amsthm}
\usepackage{ifthen}

\newcommand{\Dchaintwo}[4]{
\rule[-3\unitlength]{0pt}{8\unitlength}
\begin{picture}(14,5)(0,3)
\put(1,2){\ifthenelse{\equal{#1}{l}}{\circle*{2}}{\circle{2}}}
\put(2,2){\line(1,0){10}}
\put(13,2){\ifthenelse{\equal{#1}{r}}{\circle*{2}}{\circle{2}}}
\put(1,5){\makebox[0pt]{\scriptsize #2}}
\put(7,4){\makebox[0pt]{\scriptsize #3}}
\put(13,5){\makebox[0pt]{\scriptsize #4}}
\end{picture}}
\newlength{\mpb}

\newcommand\ord{{\operatorname{ord}}}

\newcommand{\BV}{\mathfrak{B}(V)}

\newcommand\YDG{^{\Gamma}_{\Gamma}\mathcal{YD}}

\newcommand\Z{\mathbb{Z}}

\newcommand\Ly{\mathcal{L}}

\newcommand\SupW{[\Ly]^{(\N)}}

\newcommand\SupWL{[L]^{(\N)}}

\newcommand\N{\mathbb{N}}

\renewcommand{\k}{\Bbbk}
\newcommand\G{\Gamma}
\newcommand\Gh{\widehat{\Gamma}}

\renewcommand\o{\otimes}       

\newcommand{\la}{\langle}      \newcommand{\ra}{\rangle}
  
\newcommand{\red}[1]{c_{ {#1}}}     
\newcommand{\redh}[1]{d_{{#1}}} 

\newcommand\kX{\k\la X\ra}

\newcommand\X{\la X\ra}

\newcommand\Sh[3]{\mathrm{Sh}(#1)=(#2|#3)}
\newcommand\nSh[3]{\operatorname{Sh}(#1)\neq(#2|#3)}

\numberwithin{equation}{section}

\theoremstyle{plain}
\newtheorem{thm}{Theorem}[section]

\newtheorem{lem}[thm]{Lemma}
\newtheorem{prop}[thm]{Proposition}

\newtheorem{conj}[thm]{Conjecture}

\theoremstyle{definition} 
\newtheorem{defn}[thm]{Definition}

\newtheorem{rem}[thm]{Remark}
\newtheorem{exs}[thm]{Examples}
\newtheorem{ex}[thm]{Example}

\title{On the lifting of Nichols algebras\footnote{This work is part of the author's PhD thesis written under the supervision of Professor H.-J.~Schneider.}}
\author{Michael Helbig\footnote{eMail: \texttt{michael@helbig123.de}}}
\date{\today}

\begin{document}

\maketitle


\begin{abstract}
\noindent
Nichols algebras are a fundamental building block of pointed Hopf algebras. Part of the classification program of finite-dimensional pointed Hopf algebras with the lifting method of Andruskiewitsch and Schneider  is the determination of the liftings, i.e., all possible deformations of a given Nichols algebra. Based on recent work of Heckenberger about  Nichols algebras of diagonal type  we compute explicitly the liftings of
\begin{itemize}
\item all Nichols algebras with Cartan matrix of type $A_2$, 
\item some Nichols algebras with Cartan matrix of type $B_2$, and 
\item some Nichols algebras of two Weyl equivalence classes of non-standard type 
\end{itemize}
giving  new classes of finite-dimensional pointed Hopf algebras.
\newline

\end{abstract}



\section*{Introduction}
At the moment the most promissing general method for the classification of finite-dimen\-sional pointed Hopf algebras is the lifting method developed by Andruskiewitsch and Schneider \cite{AS-p3}: Given a finite-dimensional pointed Hopf algebra $A$ with coradical $A_0=\k[\G]$ and abelian group of group-like elements $\G=G(A)$. Then we can decompose its associated graded Hopf algebra into a smash product $\operatorname{gr}(A)\cong B\# \k[\G]$ where $B$ is a braided Hopf algebra. The subalgebra of $B$ generated by its primitive elements $V:=P(B)$ is a Nichols algebra $\BV$. Now the classification is carried out in three steps:
\begin{enumerate}
\item[(1)]  Show that $B=\BV$.
\item[(2)]  Determine the structure of $\BV$.
\item[(3)]  \emph{Lifting}: Determine the liftings of $\BV$, i.e., all  Hopf algebras $A$ such that $\operatorname{gr}(A)\cong \BV \# \k[\G]$.
\end{enumerate}
Many  classification results in special situations were obtained in this way \cite{AS-p3,AS-A2,AS-CartMatr,Grana-p5,Grana32}. The most impressive result obtained by this method by Andruskiewitsch und Schneider \cite{AS-Class} is the classification of all finite-dimensional pointed Hopf algebras where the prime divisors of the order of the abelian group $\G$ are $>7$. In this case the diagonal braiding of $V$ is of Cartan type and the Hopf algebras are generalized versions of small quantum groups. The classification when the braiding is not of Cartan type or the divisors of the order of $\G$ are $\le 7$ is still an open problem. Also the case where $\G$ is not abelian is widely open and of different nature, e.g., the defining relations have another form  \cite{heckSchneidNA,andHeckSchneidSSYD}.

Concerning (2),  Heckenberger recently showed that Nichols algebras of diagonal type have a close connection to semi-simple Lie algebras, namely he introduces a Weyl groupoid \cite{HeckWeylGroupoid}, Weyl equivalence \cite{HeckWeylEquiv} and an arithmetic root system \cite{HeckArtihRootSystRank2,HeckClassArtihRootSystRank3}  for Nichols algebras. With the help of these concepts he classifies the diagonal braidings of $V$ such that the Nichols algebra $\BV$ has a finite set of PBW generators \cite{HeckClassArtihRootSyst}. Moreover, he determines the structure of all rank two Nichols algebras in terms of generators and relations \cite{HeckBV1}.

This is the starting point of our work which addresses to step (3) of the program, namely the \emph{lifting} in the cases not treated in \cite{AS-Class}. As said before, we lift Nichols algebras of diagonal type with Cartan matrix of type $A_2$ in  Theorem \ref{ExLiftingsA2}, with Cartan matrix of type $B_2$ in Theorem \ref{ExLiftingsB2}, and of non-standard type in Theorem  \ref{ThmLiftRow8}: When lifting arbitrary diagonal Nichols algebras, new phenomena occur: In the setting of \cite{AS-Class} there are only three types of defining relations, namely the Serre relations, 
the linking relations and the root vector relations. The algebraic structure in the general setting is more complicated: Firstly, the Serre relations do not play the outstanding role. Other relations are needed and sometimes the Serre relations are redundant; we give a complete answer for the Serre relations in Lemma \ref{LemLiftofGeneralRel}. Secondly, in general the lifted relations from the Nichols algebra do not remain in the group algebra.

The paper is organized as follows: In Section 1 we recall the basic notions of Nichols algebras of diagonal type, taking into account the recent developement by Heckenberger. 
Then in Section 2  the  general calculus for $q$-commutators in an arbitrary algebra of \cite{Helbig-PhD, Helbig-PBW} is presented. Section \ref{SectLyndW} contains the theory of Lyndon words, super letters and super words.
Crucial for the lifting is the knowledge of a ``good'' presentation of the Nichols algebra and its liftings, in terms  of generators and (non-redundant) relations: Section \ref{SectCharHA} repeats the result  of  \cite{Helbig-Presentation}; further some new formulas for coproducts are developed. We then formulate the main results about the liftings in Section 5. We explain our method in Section \ref{SectGeneralLiftProced}.

Throughout the paper $\k$ will be a field, all vector spaces will be over $\k$,  and all tensor products are taken over $\k$.

\section{Nichols algebras}

We want to define  Nichols algebras in the category $\YDG$ of Yetter-Drinfel'd modules over an abelian group $\G$, not necessarily finite. Our main reference is the survey article \cite[Sect.~1,2]{AS-Pointed}.

\subsection{Yetter-Drinfel'd modules of diagonal type} \label{SectYetterDrinfeldAbGr}

 The category $\YDG$ of (left-left) Yetter-Drinfel'd modules over the Hopf algebra $\k[\G]$ is the category of left $\k[\G]$-modules which are $\G$-graded vector spaces $V=\bigoplus_{g\in\G}V_g$ such that each $V_g$ is stable under the action of $\G$, i.e., 
$
h\cdot v\in V_g\quad\text{ for all }\quad h\in\G, v\in V_g.
$
The $\G$-grading is equivalent to a left $\k[\G]$-comodule structure $\delta: V\rightarrow \k[\G]\o V$: One can define $\delta$ or the other way round $V_g$ by the equivalence $\delta(v)=g\o v$ $\Longleftrightarrow$ 
$v\in V_g$ for all $g\in\G$. 
The morphisms of $\YDG$ are the $\G$-linear maps $f:V\rightarrow W$ with $f(V_g)\subset W_g$ for all $g\in \G$.

We consider the following monoidal structure on $\YDG$: If \ $V,W$$\ \in\ $$\YDG$, then also \ \mbox{$V\o W$$\ \in\ $$\YDG$} \ by
$$
g\cdot (v\o w):=(g\cdot v)\o (g\cdot w)\quad\text{and}\quad (V\o W)_g:=\bigoplus_{hk=g} V_h\o W_k
$$
for $v\in V,w\in W$ and  $g\in \G$. The \emph{braiding} in $\YDG$ is the isomorphism
$$
c=c_{V,W}: V\o W\rightarrow W\o V,\quad c(v\o w):=(g\cdot w)\o v
$$
for all $v\in V_g,$ $g\in\G$, $w\in W$. Thus every $V$$\ \in\ $$\YDG$ is a braided vector space $(V,c_{V,V})$.

We have the following important example: For a group $\G$ we denote by $\Gh$ the character group of all group homomorphisms from $\G$ to the multiplicative group $\k^{\times}$.
\begin{defn}
Let $V$$\ \in\ $$\YDG$.  If there  is a basis $x_i$, $i\in I$, of $V$ and $g_i\in\G$, $\chi_i\in \Gh$ for all $i\in I$ such that 
$$
g\cdot x_i =\chi_i(g)x_i\quad\text{and}\quad x_i\in V_{g_i},
$$
then we say $V$ is of \emph{diagonal type}.
\end{defn}
Note that if $\k$ is algebraically closed of characteristic 0 and $\G$ is finite, then any finite-dimensional $V$$\ \in\ $$\YDG$ is of diagonal type.

For the braiding we have $c(x_i\o x_j)=\chi_j(g_i)x_j\o x_i$ for $1\le i,j\le \theta$. Hence the braiding is determined by the so-called \emph{braiding matrix} of $V$
$$
(q_{ij})_{1\le i,j\le\theta}:=(\chi_j(g_i))_{1\le i,j\le\theta}.
$$

\subsection{Nichols algebras of diagonal type}\label{SectNichAlg}
Let $V$$\ \in\ $$\YDG$. $B$ is called a \emph{Nichols algebra} of $V$, if
\begin{itemize}
\item $B=\oplus_{n\ge 0}B(n)$ is a graded braided Hopf algebra in $\YDG$,
\item $B(0)\cong \k$, 
\item $P(B) = B(1)\cong V$,
\item $B$ is generated as an algebra by $B(1)$.
\end{itemize}
Any two Nichols algebras of $V$ are isomorphic, thus we write $\BV$ for ``the'' Nichols algebra of $V$. One can construct the Nichols algebra in the following way: Let $I$ denote the sum of all ideals of $T(V)$ that are generated by homogeneous elements of degree $\ge 2$ and that are also coideals. Then $\BV\cong T(V)/I$. We say the Nichols algebra $\BV$ is of \emph{diagonal type}, if $V$ is of diagonal type.

\subsection{Cartan matrices} A matrix $(a_{ij})_{1\le i,j\le\theta}\in\Z^{\theta\times\theta}$ is called a \emph{generalized Cartan matrix}
if for all $1\le i,j\le\theta$
\begin{itemize}
 \item $a_{ii}=2$,
\item $a_{ij}\le 0$ if $i\neq j$,
\item $a_{ij}=0\Rightarrow a_{ji}=0$.
\end{itemize}
Let $\BV$ be a Nichols algebra of diagonal type. Recall that $V$ resp.~$\BV$ with  braiding matrix $(q_{ij})$ is called of \emph{Cartan type}, if there is a generalized Cartan matrix $(a_{ij})$ such that
$$
q_{ij}q_{ji}=q_{ii}^{a_{ij}}.
$$
Not every Nichols algebra is of Cartan type (see Example \ref{ExBraidmatrix} and Sections \ref{LiftofNichAlgCartMatrA2}, \ref{LiftofNichAlgCartMatrB2}, \ref{LiftofNichAlgNonStandardType}), but still we have the following  \cite[Sect.~3]{HeckWeylGroupoid}: 

If $\BV$ is finite-dimensional, then the matrix $(a_{ij})$ defined for all $1\le i\neq j\le \theta$ by
$$
a_{ii}:=2\quad\text{and}\quad a_{ij}:=-\min \{r\in \mathbb N\ |\ q_{ij}q_{ji}q_{ii}^{r}=1\text{ or }(r+1)_{q_{ii}}=0\}
$$ 
is a generalized Cartan matrix fulfilling
\begin{align}\label{RemAssociateCartMatr}
 q_{ij}q_{ji}=q_{ii}^{a_{ij}}\quad\text{ or }\quad\ord q_{ii}=1-a_{ij}.
\end{align}
We call $(a_{ij})$ the  \emph{Cartan matrix associated to} $\BV$.

\subsection{Weyl equivalence} Heckenberger introduced in \cite[Sect.~2]{HeckWeylEquiv, HeckWeylGroupoid}  the notion of the \emph{Weyl groupoid} and \emph{Weyl equivalence} of Nichols algebras of diagonal type. With the help of these concepts Heckenberger classified in a series of articles \cite{HeckArtihRootSystRank2,HeckClassArtihRootSystRank3,HeckClassArtihRootSyst} all braiding matrices $(q_{ij})$ of diagonal Nichols algebras with a finite set of PBW generators. We are concerned with the list of rank 2 Nichols algebras given in the table \cite[Figure 1]{HeckWeylEquiv,HeckArtihRootSystRank2}.

We want to recall the following: For diagonal $\BV$ with braiding matrix $(q_{ij})$ we associate a \emph{generalized Dynkin diagram}: this is a graph with $\theta$ vertices, where the $i$-th vertex is labeled with $q_{ii}$ for all $1\le i\le\theta$; further, if $q_{ij}q_{ji}\neq 1$, then there is an edge between the $i$-th and $j$-th vertex labeled with $q_{ij}q_{ji}$: Thus, if $q_{ij}q_{ji}= 1$ resp.~$q_{ij}q_{ji}\neq 1$, then we have
\begin{center}\setlength{\unitlength}{1.5mm}
$\ldots$
\rule[-3\unitlength]{0pt}{8\unitlength}
 \begin{picture}(14,5)(0,3)
 \put(1,2){\circle{2}}
 \put(13,2){\circle{2}}
 \put(1,5){\makebox[0pt]{\normalsize$q_{ii}$}}
 \put(13,5){\makebox[0pt]{\normalsize$q_{jj}$}}
 \end{picture}\ \ $\ldots$ \hspace*{1.5cm}resp.\hspace*{1.5cm}
 $\ldots$ \Dchaintwo{}{\normalsize$q_{ii}$}{\normalsize$q_{ij}q_{ji}$}{\normalsize$q_{jj}$}  \ $\ldots$
\end{center}
So two Nichols algebras of the same rank $\theta$ with braiding matrix $(q_{ij})$ resp.~$(q_{ij}')$ 
have the same generalized Dynkin diagram if and only if they are \emph{twist equivalent} \cite[Def.~3.8]{AS-Pointed}, i.e., for all $1\le i,j\le\theta$ 
 $$
q_{ii}=q_{ii}'\quad\text{ and }\quad q_{ij}q_{ji}=q_{ij}'q_{ji}' .
$$

\begin{defn}
 Let  $1\le k\le \theta$  be fixed and $\BV$ finite-dimensional with braiding matrix $(q_{ij})$ and Cartan matrix $(a_{ij})$. 
We call $(q_{ij}^{(k)})$ defined by  
$$
q_{ij}^{(k)}:=q_{ij}  q_{ik}^{-a_{kj}}   q_{kj}^{-a_{ki}}  q_{kk}^{a_{ki}a_{kj}}
$$
the \emph{at the vertex $k$ reflected braiding matrix}.

Two Nichols algebras with braiding matrix $(q_{ij})$ resp. $(q_{ij}')$ are called \emph{Weyl equivalent}, if there are $m\ge 1$, $1\le k_1,\ldots,k_m\le \theta$ such that the generalized Dynkin diagrams w.r.t.~the matrices $\bigl( (\ldots(q_{ij}^{(k_1)})^{(k_2)}\ldots)^{(k_m)} \bigr)$ and $(q_{ij}')$ coincide, i.e., one gets the Dynkin diagram of $(q_{ij}')$ by successive reflections of $(q_{ij})$.
\end{defn}

\setlength{\unitlength}{1mm}
\settowidth{\mpb}{$q_0\in \k^\times \backslash \{-1,1\}$,}
\begin{ex}\label{ExBraidmatrix}
The braiding matrix 
$(q_{ij}):=\left(\begin{smallmatrix}  q & 1 \\ q^{-1}  &-1 \end{smallmatrix}\right)
$ with $q\neq 1$ has the generalized Dynkin diagram \Dchaintwo{}{$q$}{$q^{-1}$}{$-1$}\ and  associated Cartan matrix $(a_{ij})=
\left(\begin{smallmatrix} 2&-1\\-1&2 \end{smallmatrix}\right)
$ of type $A_2$, since $q_{12}q_{21}=q_{11}^{-1}$ and $\ord q_{22}=1-(-1)$. Then the at the vertex $2$ reflected braiding matrix is 
$(q_{ij}'):=(q_{ij}^{(2)})=
\left(\begin{smallmatrix}-1 & -1 \\ -q  &-1 \end{smallmatrix}\right)
$.
Its Dynkin diagram is$\!\!$ \Dchaintwo{}{$-1$}{$q$}{$-1$}\ and the associated Cartan matrix is also of type $A_2$. The two braiding matrices $(q_{ij})$ and $(q_{ij}')$ are by definition Weyl equivalent; they are twist equivalent if and only if $q=-1$. See also \cite[Figure 1]{HeckWeylEquiv,HeckArtihRootSystRank2}: row 3 if $q\neq \pm 1$ and row 2 if $q=-1$.
\end{ex}

\begin{rem}\
\begin{enumerate}
 \item Both twist equivalence and Weyl equivalence are equivalence relations, and twist equivalent Nichols algebras are
Weyl equivalent.
 \item Weyl equivalent Nichols algebras have the same dimension and Gel'fand-Kirillov dimension \cite[Prop.~1]{HeckWeylEquiv}, but can have different associated Cartan matrices. 
If the whole Weyl equivalence class has the same Cartan matrix, then 
the Nichols algebras of this class are called of \emph{standard} type \cite{AAGeneric,angionoStandard}. 
\end{enumerate}
\end{rem}

\begin{exs} Let $\BV$ be of rank 2.  Then two Nichols algebras are Weyl equivalent if and only if their generalized Dynkin diagrams appear in the same row of \cite[Figure 1]{HeckWeylEquiv,HeckArtihRootSystRank2} and can be presented with the same set of fixed parameters \cite{HeckWeylEquiv}.

\begin{enumerate}
 \item $\BV$  is of standard type, if and only if it appears in the rows 1--7, 11 or 12  of \cite[Figure 1]{HeckWeylEquiv,HeckArtihRootSystRank2}. The Cartan matrices are
\begin{itemize}
 \item  $\left(\begin{smallmatrix} 2&0\\0&2 \end{smallmatrix}\right)
$ of type $A_1\times A_1$ of row 1,
\item   $\left(\begin{smallmatrix} 2&-1\\-1&2\end{smallmatrix}\right)
$ of type $A_2$ of rows 2 and 3,
\item   $\left(\begin{smallmatrix} 2&-2\\-1&2 \end{smallmatrix}\right)
$  of type $B_2$ of  rows 4--7, and
\item  $\left(\begin{smallmatrix}2&-3\\-1&2 \end{smallmatrix}\right)
$ of type $G_2$ of rows 11 and 12.
\end{itemize}
All Nichols algebras of type $A_1\times A_1$, $A_2$ and some of $B_2$ are lifted in Sections \ref{LiftofNichAlgCartMatrA1xA1}, \ref{LiftofNichAlgCartMatrA2}, \ref{LiftofNichAlgCartMatrB2}.

\item In the non-standard Weyl equivalence class of row 8 of \cite[Figure 1]{HeckWeylEquiv,HeckArtihRootSystRank2}  the  Cartan matrices
\begin{itemize}
\item 
$
\left(\begin{smallmatrix} 2&-2\\-2&2 \end{smallmatrix}\right)
$ of \Dchaintwo{}{$-\zeta ^{-2}$}{$-\zeta ^3$}{$-\zeta ^2$}, 
\item 
$
\left(\begin{smallmatrix} 2&-2\\-1&2 \end{smallmatrix}\right)
$ of type $B_2$ of \Dchaintwo{}{$-\zeta ^{-2}$}{$\zeta ^{-1}$}{$-1$},\ 
 \Dchaintwo{}{$-\zeta ^2$}{$-\zeta $}{$-1$}, and
\item  
$
\left(\begin{smallmatrix}2&-3\\-1&2 \end{smallmatrix}\right)
$ of type $G_2$  of \Dchaintwo{}{$-\zeta ^3$}{$\zeta $}{$-1$},\
\Dchaintwo{}{$-\zeta ^3$}{$-\zeta ^{-1}$}{$-1$}
\end{itemize}
appear. These Nichols algebras are lifted in Section \ref{LiftofNichAlgNonStandardType}. The same Cartan matrices appear in row 9, where we lift the Nichols algebras corresponding to the last two Dynkin diagrams.
\end{enumerate}
\end{exs}

    \section{$q$-commutator calculus}\label{SectqCommutCalc}
In this section let $A$ denote an arbitrary algebra over a field $\k$ of characteristic $\operatorname{char}\k=p\ge 0$.  The main result of this chapter is Proposition \ref{PropqCommut}, which states important $q$-commutator formulas in an arbitrary algebra.

\subsection{$q$-calculus} For every $q\in \k$ we define for $n\in\N$ and $0\le i\le n$ the \emph{$q$-numbers} $(n)_q := 1+q+q^2+\ldots+q^{n-1}$, the \emph{$q$-factorials} 
$(n)_q ! := (1)_q (2)_q \ldots (n)_q,$ 
and the \emph{$q$-binomial coefficients} $\tbinom{n}{i}_q:=\frac{(n)_q !}{(n-i)_q !(i)_q !}.$
Note that the latter right-handside is well-defined since it is a polynomial over $\Z$ evaluated in $q$.
We denote the \emph{multiplicative order} of any $q\in  \k^{\times}$ by $\ord q$.
If $q\in \k^{\times}$ and $n>1$, then 
\begin{align}\label{qBinomCoeffZero}
\binom{n}{i}_q =0\text{ for all }1\le i\le n-1 \Longleftrightarrow 
\begin{cases}
 \ord q = n ,&\mbox{if }\operatorname{char}\k=0\\
p^k\ord q =n\text{ with }k\ge 0 ,&\mbox{if }\operatorname{char}\k=p>0,
\end{cases}
\end{align}
see \cite[Cor. 2]{Radford}. Moreover for $1\le i\le n$ there are the \emph{$q$-Pascal identities}
\begin{align}\label{PascalDreieck}
q^i \binom{n}{i}_q + \binom{n}{i-1}_q = 
\binom{n}{i}_q +q^{n+1-i}\binom{n}{i-1}_q =\binom{n+1}{i}_q,
\end{align}
and the \emph{$q$-binomial theorem}: For $x,y\in A$ and $q\in \k^{\times}$ with $yx=qxy$ we have
\begin{align}\label{qBinomThm}
(x+y)^n=\sum_{i=0}^{n} \tbinom{n}{i}_q x^i y^{n-i}.
\end{align}
Note that for $q=1$ these are the usual notions.

\subsection{$q$-commutators}\label{qCommutExkX}\label{qCommutEx}
Let $\theta\ge 1$, $X=\{ x_1,\ldots,x_{\theta}\}$, $\X$ the free monoid and $A=\kX$ the free $\k$-algebra. For an abelian group $\G$ let $\widehat{\G}$ be the character group, $g_1,\ldots,g_{\theta}\in\G$ and $\chi_1,\ldots,\chi_{\theta}\in \widehat{\G}$. If we define the two monoid maps 
$$
\deg_{\G}:\X\rightarrow \G,\ \deg_{\G}(x_i):=g_i\quad\text{and}\quad \deg_{\Gh}:\X\rightarrow \Gh,\ \deg_{\Gh}(x_i):=\chi_i,
$$ 
for all $1\le i\le\theta$, then $\kX$ is $\G$- and $\Gh$-graded. 
Let $a\in \kX$ be $\G$-homogeneous and $b\in \kX$ be $\widehat{\G}$-homogeneous. We set 
$$
g_a:=\deg_{\G}(a), \quad \chi_b:=\deg_{\Gh}(b),\quad\text{and}\quad q_{a,b}:=\chi_b(g_a).
$$ 
Further we define $\k$-linearly on $\kX$ the \emph{$q$-commutator}
\begin{align}\label{qCommutExkXDefn}
 [a,b] := [a,b]_{q_{a,b}}.
\end{align}
Note that $q_{a,b}$ is a bicharacter on the homogeneous elements 
and depends only on the values 
$$
q_{ij}:=\chi_j(g_i)\text{ with }1\le i,j\le\theta.$$
For example $[x_1,x_2]=x_1x_2-\chi_2(g_1)x_2x_1=x_1x_2-q_{12}x_2x_1$. Further if $a,b$ are $\Z^{\theta}$-homogeneous they are both $\G$- and $\Gh$-homogeneous. In this case we can build iterated $q$-commutators, like $\bigl[x_1,[x_1,x_2]\bigr] =x_1[x_1,x_2]-\chi_1\chi_2(g_1)[x_1,x_2]x_1= x_1[x_1,x_2]-q_{11}q_{12}[x_1,x_2]x_1$.

\begin{prop}\emph{\cite[Prop.~1.2]{Helbig-PBW}}\label{PropqCommut} For all homogeneous $a,b,c\in \kX$ and $r\ge 1$ we have:

\emph{(1) $q$-derivation properties:}
$[a,bc]=[a,b] c + q_{a,b} b [a,c], \
   [ab,c]=a [b,c] + q_{b,c} [a,c] b.$

\emph{(2) $q$-Jacobi identity:}
$\bigl[[a,b], c\bigr] =\bigl[a,[b,c]\bigr] -q_{a,b} b [a,c]+ q_{b,c}[a,c] b.$

\emph{(3) $q$-Leibniz formulas:}
\begin{align*}
  [a,b^r] &= \sum_{i=0}^{r-1} q_{a,b}^i \tbinom{r}{i}_{q_{b,b}} 
                           b^i \bigl[\ldots\bigl[[a,\underbrace{b] , b\bigr]\ldots,b}_{r-i}\bigr],\\
  [a^r,b] &= \sum_{i=0}^{r-1} q_{a,b}^i \tbinom{r}{i}_{q_{a,a}} 
                      \bigl[\underbrace{a,\ldots\bigl[a,[a}_{r-i},b]\bigr]\ldots\bigr]a^i.
\end{align*}

\emph{(4) restricted $q$-Leibniz formulas:} If $\operatorname{char} \k=0$ and $\ord q_{b,b} = r$ resp.~$\ord q_{a,a} = r$, or $\operatorname{char} \k =p >0$ and $p^k \ord q_{b,b}\zeta =r$ resp.~$p^k\ord q_{a,a} = r$, then
\begin{align*}
 [a,b^r] = \bigl[\ldots\bigl[[a,\underbrace{b],b\bigr]\ldots,b}_{r}\bigr]\ \text{ resp. }\ 
 [a^r,b] = \bigl[\underbrace{a,\ldots\bigl[a,[a}_{r},b]\bigr]\ldots\bigr].
\end{align*}
\end{prop}

\section{Lyndon words and $q$-commutators}\label{SectLyndW}

In this section we recall the theory of Lyndon words \cite{Lothaire,Reut} as far as we are concerned and then introduce the notion of super letters and super words \cite{KhPBW}.

\subsection{Words and the lexicographical order}
Let $\theta\ge 1$, $X=\{x_1,x_2,\ldots,x_{\theta}\}$ be a finite totally ordered set by $x_1<x_2<\ldots <x_{\theta}$, and $\X$ the free monoid; we think of $X$ as an alphabet and of $\X$ as the words in that alphabet including the empty word $1$. For a word $u=x_{i_1}\ldots x_{i_n}\in\X$ we define $\ell(u):=n$ and call it the \emph{length} of $u$. 

The \emph{lexicographical order} $\le$ on $\X$ is defined for $u,v\in\X$ by $u<v$ if and only if 
either $v$ begins with $u$, i.e., $v=uv'$ for some $v'\in\X\backslash\{1\}$, or if there are $w,u',v'\in \X$, $x_i,x_j\in X$ 
such that $u=wx_iu'$, $v=wx_jv'$ and $i<j$. E.g., $x_1<x_1x_2<x_2$. 

\subsection{Lyndon words and the Shirshov decomposition}
A word $u\in\X$ is called a \emph{Lyndon word} if $u\neq 1$ and $u$ is smaller than any of its proper endings, i.e., for all $v,w\in\X\backslash\{1\}$ such that $u=vw$ we have $u<w$. We denote by 
$$
\Ly:=\{u\in\X\,|\, u \text{ is a Lyndon word}\}
$$ the set of all Lyndon words. For example $X\subset\Ly$, but $x_i^n\notin \Ly$ for all $1\le i\le \theta$ and $n\ge 2$. 
Also $x_1x_2$, $x_1x_1x_2$, $x_1x_2x_2$,  $x_1x_1x_2x_1x_2\in\Ly$.

For any $u\in\X\backslash X$ we call the decomposition $u=vw$ with $v,w\in \X\backslash\{1\}$ such that $w$ is the minimal (with respect to the lexicographical order) ending  the \emph{Shirshov decomposition} of the word $u$. We will write in this case $$\Sh{u}{v}{w}.$$ E.g., $\Sh{x_1x_2}{x_1}{x_2}$, $\Sh{x_1x_1x_2x_1x_2}{x_1x_1x_2}{x_1x_2}$, $\nSh{x_1x_1x_2}{x_1x_1}{x_2}$.
If $u\in\Ly\backslash X$, this is equivalent to $w$ is the longest proper ending of $u$ such that $w\in\Ly$.

\begin{defn}
We call a subset $L\subset \Ly$  \emph{Shirshov closed} 
if
  $X\subset L$,
  and for all $u\in L$ with $\Sh{u}{v}{w}$ also $v,w\in L$.
\end{defn}

For example $\Ly$ is Shirshov closed, and if $X=\{x_1,x_2\}$, then $\{x_1,x_1x_1x_2,x_2\}$ is  not Shirshov closed,  whereas $\{x_1,x_1x_2,x_1x_1x_2,x_2\}$ is.

\subsection{Super letters and super words} Let the free algebra $\kX$ be graded as in Section \ref{qCommutExkX}. For any $u\in\Ly$ we define recursively on $\ell(u)$ the map 
\begin{align}\label{DefnSuperLett}
[\,.\,]:\Ly\rightarrow\kX,\quad u \mapsto [u].
\end{align}
If $\ell(u)=1$, then set $[x_i]:=x_i$ for all $1\le i\le\theta$. Else if $\ell(u)>1$ and $\Sh{u}{v}{w}$ we define $[u]:=\bigl[[v],[w]\bigr]$. This map is well-defined since inductively all $[u]$ are $\Z^{\theta}$-homogeneous such that we can build iterated $q$-commutators; see Section \ref{qCommutExkX}. The elements $[u]\in\kX$ with $u\in\Ly$ are called \emph{super letters}. E.g. $[x_1x_1x_2x_1x_2]=\bigl[[x_1x_1x_2],[x_1x_2]\bigr]=\bigl[[x_1,[x_1,x_2]],[x_1,x_2]\bigr]$.
If $L\subset \Ly$ is Shirshov closed then the subset of $\kX$
$$
[L]:=\bigl\{[u]\,\big|\, u\in L\bigr\}
$$ 
is a set of iterated $q$-commutators. Further 
$
[\Ly]=\bigl\{[u]\,\big|\, u\in\Ly\bigr\}
$
is the set of all super letters and the map 
$[\,.\,]:\Ly\rightarrow[\Ly]$ is a bijection, which follows from \cite[Lem.~2.5]{Helbig-Presentation}. Hence we can define an order $\le$ of the super letters $[\Ly]$ by 
$$
[u]< [v]:\Leftrightarrow u<v,
$$ thus $[\Ly]$ is a new alphabet containing the original alphabet $X$; so the name ``letter'' makes sense. Consequently, products of super letters are called \emph{super words}. We denote 
$$
[\Ly]^{(\N)}:=\bigl\{[u_1]\ldots [u_n]\,\bigl|\,n\in\N,\, u_i\in\Ly\bigr\}
$$
the subset of $\kX$ of all super words. Any super word has a unique factorization in super letters \cite[Prop.~2.6]{Helbig-Presentation}, hence we can define the lexicographical order on $\SupW$, as defined above on regular words. We denote it also by $\le$.

\subsection{A well-founded ordering of super words}\label{SectWellFoundOrder} The \emph{length} of a super word $U=[u_1][u_2]\ldots[u_n]\in\SupWL$ is defined as 
$
\ell(U):=\ell(u_1u_2\ldots u_n).
$
\begin{defn}
For $U, V\in \SupW$ we define $U\prec V$ by 
\begin{itemize}
\item $\ell(U)<\ell(V)$, or
\item $\ell(U)=\ell(V)$  and $U>V$ lexicographically in $\SupW$.
\end{itemize}
\end{defn}
This defines a total ordering of $\SupW$ with minimal element $1$. As $X$ is assumed to be finite, there are only finitely many super letters of a given length. Hence every nonempty subset of $\SupW$ has a minimal element, or equivalently, $\preceq$ fulfills the descending chain condition: $\preceq$ is \emph{well-founded}. This makes way for inductive proofs on $\preceq$.

  \section{A class of pointed Hopf algebras}\label{SectCharHA}
In this chapter we deal with a special class of pointed Hopf algebras. Let us recall the notions and results of \cite[Sect.~3]{KhPBW}: A Hopf algebra $A$ is called a \emph{character Hopf algebra} if it is generated as an algebra by 
elements $a_1,\ldots ,a_{\theta}$ and an abelian group $G(A)=\G$ of all group-like elements such that  for all $1\le i\le \theta$ there are $g_i\in\G$ and $\chi_i\in\Gh$ with
\begin{align*}
\Delta(a_i)=a_i\o 1 +g_i\o a_i\qquad \text{and}\qquad
 ga_i = \chi_i(g)a_i g.
\end{align*}
As mentioned in the introduction this covers a wide class of examples of Hopf algebras. 

\begin{thm}\emph{\cite[Thm.~3.4]{Helbig-Presentation}} 
\label{PropIdealCharHopfAlg}
If $A$ is a character Hopf algebra, then
$$
A\cong (\kX \# \k[\G])/I,
$$
where the smash product $\kX\#\k[\G]$ and the ideal $I$ are constructed in the following way:
\end{thm}

\subsection{The smash product $\kX\# \k[\G]$}
Let $\kX$ be $\G$- and $\Gh$-graded as in Section \ref{qCommutExkX}, and $\k[\G]$ be endowed with the usual bialgebra structure $\Delta(g)=g\o g$ and $\varepsilon(g)=1$ for all $g\in\G$. 
Then we define 
$$
g\cdot x_i := \chi_i(g)x_i,\ \text{ for all }1\le i\le\theta.
$$
In this case, $\kX$ is a $\k[\G]$-module algebra and we calculate $gx_i=\chi_i(g)x_ig$, $gh=hg=\varepsilon(g)hg$ in $\kX\# \k[\G]$.
Thus  $x_i\in(\kX\# \k[\G])^{\chi_i}$ and $\k[\G]\subset (\kX\# \k[\G])^\varepsilon$ and in this way 
$
\kX\# \k[\G]=\oplus_{\chi\in\Gh}(\kX\# \k[\G])^\chi.
$
This $\Gh$-grading extends the $\Gh$-grading of $\kX$ in Section \ref{qCommutExkX} to $\kX\# \k[\G]$. 
Further $\kX\# \k[\G]$ is a Hopf algebra with structure determined for all $1\le i \le\theta$ and $g\in\G$ by 
\begin{align*}
 \Delta(x_i):=x_i\o 1 + g_i\o x_i\qquad\text{and}\qquad  \Delta(g):=g\o g.
\end{align*}

\subsection{Ideals associated to Shirshov closed sets}\label{SectIdealOfCharHA}
In this subsection we fix a Shirshov closed $L\subset\Ly$. 
We want to introduce the following notation 
for an $a\in \kX\# \k[\G]$ and  
$W\in\SupW$: 
We will write $a\prec_L W$ (resp. $a\preceq_L W$), if 
$a$ is a linear combination of 
\begin{enumerate}
 \item[\textbullet] $U\in\SupWL$ with $\ell(U)=\ell(W)$, $U>W$ (resp. $U\ge W$), and
\item[\textbullet]  $Vg$ with $V\in\SupWL$, $g\in\G$,  $\ell(V)<\ell(W)$.
\end{enumerate}

Furthermore,  we set for each $u\in L$
either $N_u:=\infty$ or $N_u:=\ord q_{u,u}$  (resp.~$N_u:=p^k\ord q_{u,u}$ with $k\ge 0$ if $\operatorname{char} \k=p>0$) and we want to distinguish the following two sets of words depending on $L$:
\begin{align*}
C(L) &:= \bigl\{w\in\X\backslash L \ |\ \exists u,v\in L:  w=uv,\ u<v,\  \text{and}\ \Sh{w}{u}{v}\bigr\},\\
D(L) &:= \bigl\{u\in L \ |\ N_u<\infty\}.
\end{align*}
Note that $C(L)\subset \Ly$ and $D(L)\subset L\subset \Ly$ are sets of Lyndon words. For example, if $L=\{x_1,x_1x_1x_2,x_1x_2,x_2\}$, then  $C(L)=\{x_1x_1x_1x_2$, $x_1x_1x_2x_1x_2$, $x_1x_2x_2\}$.

Moreover, let $\red{w} \in (\kX\# \k[\G])^{\chi_{w}}$ for all $w\in C(L)$ such that $\red{w}\prec_L [w]$;
and let $\redh{u} \in (\kX\# \k[\G])^{\chi_u^{N_u}}$ for all $u\in D(L)$ such that $\redh{u}\prec_L [u]^{N_u}$. 
Then let $I$ be the $\Gh$-homogeneous ideal of $\kX\# \k[\G]$ generated by the following elements: 
\begin{align}
[w]         - \red{w}&    &&\text{for all } w\in C(L) ,\label{ThmPBWCritIdluv}\\
 [u]^{N_u}   - \redh{u}&    &&\text{for all }  u\in D(L).\label{ThmPBWCritIdluN}
\end{align}

    \subsection{Calculation of coproducts}\label{SectCoproducts}
Let in this section $\operatorname{char}\k=0$. 
For any $g\in\G,\chi\in\Gh$ we set 
$$
P_g^\chi:=P_g^\chi(A):=P_{1,g}(A)\cap A^\chi=\{a\in A\ |\ \Delta(a)=a\o 1+g\o a,\, ga=\chi(g)ag \}.
$$
Although the following calculations are for $\kX\# \k[\G]$, we can use the results in any character Hopf algebra $A$ by the canonical Hopf algebra map $\kX\# \k[\G]\rightarrow A$.  Assume again the situation of Section \ref{qCommutEx}.

\begin{lem} \label{LemSkewPrimRootSerre}
Let $1\le i< j\le\theta$ and $r\ge 1$.
\begin{enumerate}
\item[\emph{(1)}] If $\ord q_{ii}=N$, then $x_i^N\in P_{g_i^N}^{\chi_i^N}$.
\item[\emph{(2)}] If $q_{ij}q_{ji}=q_{ii}^{-(r-1)}$ and $r \le \ord q_{ii}$, then $[x_i^r x_j]\in P_{g_i^r g_j}^{\chi_i^r \chi_j}$.
\item[\emph{(3)}] If $q_{ij}q_{ji}=q_{jj}^{-(r-1)}$ and $r \le \ord q_{jj}$, then $[x_ix_j^r]\in P_{g_ig_j^r}^{\chi_i \chi_j^r}$.
\end{enumerate}
\end{lem}
\begin{proof} (1) We have $(g_i\o x_i)(x_i\o 1)=q_{ii}(x_i\o 1)(g_i\o x_i)$ hence by Eq.~\eqref{qBinomThm} we obtain the claim.
For (2) and (3) see \cite[Lem.~A.1]{AS-CartMatr}.
\end{proof}
Next we want to examine certain coproducts in the special case when $q_{ii}=-1$ for a $1\le i\le\theta$. Note that in the following two Lemmata we could write more generally $i$ and $j$ with $1\le i< j\le\theta$ instead of $1$ and $2$:
\begin{lem}  \label{LemSkewZ}
Let 
$\ord q_{12,12}=N$.
\begin{enumerate}
 \item[\emph{(1)}] If $q_{22}=-1$, we have for the quotient $(\kX\# \k[\G])/(x_2^2)$
\begin{align*}
 \Delta\bigl([x_1x_2]^N\bigr)=[x_1&x_2]^N\o 1  +  g_{12}^N\o [x_1x_2]^N\\
 &+    q_{2,12}^{N-1}(1-q_{12}q_{21}) [x_1(x_1x_2)^{N-1}] g_2\o x_2.
\end{align*}

\item[\emph{(2)}] If $q_{11}=-1$, we have for the quotient $(\kX\# \k[\G])/(x_1^2)$ 
\begin{align*}
 \Delta\bigl([x_1x_2]^N\bigl)=[x_1&x_2]^N\o 1 +  g_{12}^N\o [x_1x_2]^N\\
   &+   q_{1,12}^{N-1}(1-q_{12}q_{21}) x_1 g_{12}^{N-1} g_2\o [(x_1x_2)^{N-1} x_2]. 
\end{align*}

\end{enumerate}
\end{lem}
\begin{proof}We calculate  directly in  $\kX\# \k[\G]$
$$
\Delta([x_1x_2])=[x_1x_2]\o 1 + (1-q_{12}q_{21})x_1g_2\o x_2 +g_{12}\o [x_1x_2].
$$
For $\alpha:=(1-q_{12}q_{21})$, $q:=q_{12,12}$, $U:=[x_1x_2]\o 1$, $V:=\alpha x_1g_2\o x_2$ and $W:=g_{12}\o [x_1x_2]$ we have $WU=qUW$ and
\begin{align*}
VU-qUV &= \alpha q_{2,12} [x_1x_1x_2]g_2\o x_2,\\
WV-qVW &= \alpha q_{12,1} x_1g_{12}g_2 \o [x_1x_2x_2].
\end{align*}
We further set for $r\ge 1$
\begin{align*}
[V]&:=V,&[VU^r]&:=\alpha q_{2,12}^r [x_1(x_1x_2)^r]
g_2\o x_2,\\
[W]&:=W,& [W^rV]
&:=\alpha q_{1,12}^r x_1 g_{12}^r g_2\o [(x_1x_2)^r x_2].
\end{align*}

(1) We have $[x_1x_2x_2]=[x_1,x_2^2]=0$ by the restricted $q$-Leibniz formula and $x_2^2=0$. Hence  $WU=qUW$ and $WV=qVW$. By Eq.~\eqref{qBinomThm} we have
$$
\Delta([x_1x_2]^r)=(U+V+W)^r=(U+V)^r + W^r.
$$
We state for $r\ge 1$
$$
(U+V)^r=U^r + \sum_{i=0}^{r-1}\tbinom{r}{i}_q U^i[VU^{(r-1)-i}],
$$
from where the claim follows. This we
prove by induction on $r$: For $r=1$ the claim is true. By induction assumption
\begin{align*}
 (&U+V)^{r+1}=(U+V)^r(U+V)\\
 &= U^{r+1}+\sum_{i=0}^{r-1}\tbinom{r}{i}_q U^i[VU^{(r-1)-i}]U+U^rV+\sum_{i=0}^{r-1}\tbinom{r}{i}_q U^i[VU^{(r-1)-i}]V,
\end{align*}
where the last sum is zero since $[VU^{(r-1)-i}]V=\ldots\o x_2^2=0$ for all $0\le i\le r-1$. Further
\begin{align*}
[VU^{(r-1)-i}]U&=\alpha q_{2,12}^{(r-1)-i} [x_1(x_1x_2)^{(r-1)-i}] g_2[x_1x_2]\o x_2 \\
&=\alpha q_{2,12}^{r-i} \bigl([x_1(x_1x_2)^{r-i}][x_1x_2]\\
&\qquad\qquad+ q_{1,12}q^{(r-1)-i} [x_1x_2][x_1(x_1x_2)^{(r-1)-i}]  \bigr)g_2\o x_2   \\
&=[VU^{r-i}]+q^{r-i}U[VU^{(r-1)-i}].
\end{align*}
Thus $(U+V)^{r+1}=$ 
\begin{align*}
&= U^{r+1}+\sum_{i=0}^{r-1}\tbinom{r}{i}_q U^i [VU^{r-i}] +U^rV
 +\sum_{i=0}^{r-1}\tbinom{r}{i}_q q^{r-i}  U^{i+1}[VU^{(r-1)-i}] \\
    &=U^{r+1}+\sum_{i=0}^{r}\bigl(\tbinom{r}{i}_q+ \tbinom{r}{i-1}_q q^{r+1-i}\bigr) U^i [VU^{r-i}],
\end{align*}
by shifting the index of the second sum. Using Eq.~\eqref{PascalDreieck} we get the desired formula. 

(2) is proven analogously with the formula $(V+W)^r=W^r + \sum_{i=0}^{r-1}\tbinom{r}{i}_q [W^{(r-1)-i}V]V^i.
$
\end{proof}
A direct computation in $\kX\# \k[\G]$ shows that
\begin{align*}
\Delta&([x_1x_1x_2x_1x_2]) =  [x_1x_1x_2x_1x_2] \otimes 1   +   g_1^3g_2^2 \otimes [x_1x_1x_2x_1x_2]\\
      &\quad + \alpha  [x_1x_1x_2] g_1g_2\otimes [x_1x_2]\\
      &\quad + (1-q_{12}q_{21}) \big(q_{21}q_{22}\beta [x_1x_1x_1x_2]+\alpha  [x_1x_1x_2]x_1 \big) g_2 \otimes x_2\\
      &\quad + (1-q_{12}q_{21})(1- q_{11}q_{12}q_{21}) x_1^2g_1g_2^2\\
      &\qquad\qquad\qquad \otimes
	        \big( q_{11}q_{21}(1+q_{11}-q_{11}^3 q_{12}^2 q_{21}^2 q_{22})[x_1x_2x_2] + \alpha x_2 [x_1x_2] \big) \\
      &\quad + q_{21}(1-q_{12}q_{21})^2(1-q_{11}q_{12}q_{21}) (1- q_{11}^2 q_{12}^2 q_{21}^2 q_{22})  x_1^3 g_2^2\otimes x_2^2\\
      &\quad + x_1 g_1^2g_2^2 \otimes  \big(\gamma [x_1x_2]^2 + q_{11}^2q_{21}(1-q_{12}q_{21}) [x_1x_1x_2x_2] \big),
\end{align*}
with
\begin{align*}
\alpha &:= (2)_{q_{11}}q_{11}q_{12}q_{21}q_{22}(1  - q_{11}q_{12}q_{21}) + 1- q_{11}^4 q_{12}^3 q_{21}^3 q_{22}^2 , \\
\beta     &:= 1-q_{11}q_{12}q_{21} - q_{11}^2 q_{12}^2 q_{21}^2 q_{22},\\
\gamma     &:= q_{11}^2 q_{21}q_{12}(1-q_{12}q_{21})(q_{22}-q_{11})\\
        &\qquad\qquad\qquad +(2)_{q_{11}} (1- q_{11}q_{12}q_{21}) (1 - q_{11}^3 q_{12}^2 q_{21}^2 q_{22}).
\end{align*}
\begin{lem}\label{LemSkew11212} Let $q_{22}=-1$.  Then
\begin{align*}
\alpha = (3)_{q_{12,12}} (1-q_{11}^2q_{12}q_{21}),\quad
\beta    = (3)_{q_{12,12}}, \quad
\gamma   = (2)_{q_{11}}(3)_{q_{12,12}}(1-q_{11}^2q_{12}q_{21}).
\end{align*}
As a consequence we have the following:
\begin{enumerate}
\item[\emph{(1)}] If $\ord q_{12,12}=3$, then 
\begin{align*}
\Delta&([x_1x_1x_2x_1x_2]) =  [x_1x_1x_2x_1x_2] \otimes 1   +   g_1^3g_2^2 \otimes [x_1x_1x_2x_1x_2]\\
      &\quad + (1-q_{12}q_{21})(1- q_{11}q_{12}q_{21}) x_1^2g_1g_2^2\\
      &\qquad\qquad\qquad \otimes
	         q_{11}q_{21}(1+q_{11}-q_{11}^3 q_{12}^2 q_{21}^2 q_{22})[x_1x_2x_2]  \\
      &\quad + q_{21}(1-q_{12}q_{21})^2(1-q_{11}q_{12}q_{21}) (1- q_{11}^2 q_{12}^2 q_{21}^2 q_{22})  x_1^3 g_2^2\otimes x_2^2\\
      &\quad + x_1 g_1^2g_2^2 \otimes   q_{11}^2q_{21}(1-q_{12}q_{21}) [x_1x_1x_2x_2] .
\end{align*}
Hence  $[x_1x_1x_2x_1x_2]\in 
P_{g_1^3g_2^2}^{\chi_1^3\chi_2^2}$ in the quotient $(\kX\# \k[\G])/(x_2^2)$.

\item[\emph{(2)}] If $q_{12}q_{21}=q_{11}^{-2}$ and $\ord q_{11}=3$, then
\begin{align*}
\Delta&([x_1x_1x_2x_1x_2]) =  [x_1x_1x_2x_1x_2] \otimes 1   +   g_1^3g_2^2 \otimes [x_1x_1x_2x_1x_2]\\
      &\quad + (1-q_{12}q_{21}) q_{21}q_{22}\beta [x_1x_1x_1x_2] g_2 \otimes x_2\\
      &\quad + q_{21}(1-q_{12}q_{21})^2(1-q_{11}q_{12}q_{21}) (1- q_{11}^2 q_{12}^2 q_{21}^2 q_{22})  x_1^3 g_2^2\otimes x_2^2\\
      &\quad + x_1 g_1^2g_2^2 \otimes   q_{11}^2q_{21}(1-q_{12}q_{21}) [x_1x_1x_2x_2].
\end{align*}
Hence $[x_1x_1x_2x_1x_2]\in P_{g_1^3g_2^2}^{\chi_1^3\chi_2^2}$ in the quotients $(\kX\# \k[\G])/(x_2^2,\, [x_1x_1x_1x_2])$\\ or $(\kX\# \k[\G])/(x_1^3,\, [x_1x_1x_2x_2])$.
\end{enumerate}
\end{lem}
\begin{proof} This is also a straightforward calculation using the following identities: Since $q_{22}=-1$, we have $[x_1x_2x_2]=[x_1,x_2^2]$ by the restricted $q$-Leibniz formula of Proposition \ref{PropqCommut}, thus $[x_1x_1x_2x_2]=\bigl[x_1,[x_1x_2x_2]\bigr]=\bigl[x_1,[x_1,x_2^2]\bigr]$. So we see that both are zero if $x_2^2=0$.  If $\ord q_{11}=3$ analogously $[x_1x_1x_1x_2]=[x_1^3,x_2]=0$ for $x_1^3=0$.
\end{proof}

We want to state some basic combinatorics on the $g_i$'s and $\chi_i$'s for later reference:

\begin{lem}\label{LemQijCombinatorics}
 Let $1\le i\neq j\le\theta$, $1<N:=\ord q_{ii}<\infty$, and $r\in\Z$. Then:
\begin{enumerate}
\item[\emph{(1)}] $\chi_i^{N}\neq \chi_i$.
\item[\emph{(2)}] If $q_{jj}\neq 1$, then $\chi_i^{N}\neq \chi_j$ or $g_i^N\neq g_j$.
\item[\emph{(3)}] If $\chi_i^{N}= \varepsilon$, then $q_{ji}^N=1$. Especially, if $\chi_i^{2}= \varepsilon$, then $q_{ji}=\pm 1$.
\item[\emph{(4)}] If $q_{ij}q_{ji}=q_{ii}^{-(r-1)}$ and $q_{jj}\neq 1$, then $\chi_i^r\chi_j \neq \chi_i$. 
\item[\emph{(5)}] If $q_{ii}^r\neq 1$, then $\chi_i^r\chi_j \neq\chi_j$. 
\item[\emph{(6)}] If $q_{ij}q_{ji}=q_{ii}^{-(r-1)}$ and $\chi_i^r\chi_j =\varepsilon$, then
$$
\begin{pmatrix}
 q_{ii} & q_{ij}\\
q_{ji} & q_{jj}
\end{pmatrix}
=
\begin{pmatrix}
 q_{ii} & q_{ii}^{-r}\\
 q_{ii} & q_{ii}^{-r}
\end{pmatrix}.
$$
Especially, if $q_{jj}=-1$, then $q_{ii}^r=-1$ and $N$ is even.
\end{enumerate}
\end{lem}
\begin{proof}
(1) Assume $\chi_i^{N}= \chi_i$. Hence $q_{ii}^{N-1}=1$, a contradiction.\\
(2)  If  $\chi_i^{N}= \chi_j$ and $g_i^N=g_j$, then $1=q_{ii}^N=q_{ij}$, $q_{ji}^N=q_{jj}$, $1=q_{ii}^N=q_{ji}$ and $q_{ij}^N=q_{jj}$. Hence $q_{jj}=q_{ij}=q_{ji}=1$.\\ 
(3) is clear.\\
(4) If $\chi_i^r\chi_j =\chi_i$, then $q_{ii}^{r-1} q_{ij}=1$, $q_{ji}^{r-1} q_{jj}=1$. We deduce $q_{ji}=q_{jj}=1$.\\
(5) If $\chi_i^r\chi_j =\chi_j$, then $q_{ii}^r=1$.\\
(6) We have $q_{ii}^{r} q_{ij}=1$, $q_{ji}^{r} q_{jj}=1$. Now the assumption implies the claim.
\end{proof}


\section{Lifting}\label{SectLift}


We proceed as in \cite{AS-p3,AS-A2}: In this chapter let  $\operatorname{char}\k=0$ and  $A$  be a finite-dimensional pointed Hopf algebra with abelian group of group-like elements $G(A)=\G$ and assume  that the associated graded Hopf algebra with respect to the coradical filtration  is
$$
\operatorname{gr}(A)\cong \BV\#\k[\G],
$$ 
where $V$ is of diagonal type of dimension $\dim_\k V=\theta$ with basis $X=\{x_1,\ldots,x_\theta\}$. We will always identify $T(V)\cong\kX$. It is $\dim_\k A=\dim_\k \operatorname{gr}(A)=\dim_\k\BV \cdot |\G|$. In particular $\BV$ is finite-dimensional and we can associate a Cartan matrix as in Definition \ref{RemAssociateCartMatr}. 
\begin{defn}
 In this situation we say that $A$ is a \emph{lifting} of the Hopf algebra $\BV\#\k[\G]$, or simply of the Nichols algebra $\BV$.
\end{defn} 
By \cite[Lem.~5.4]{AS-p3}, we have that 
\begin{align}\label{LemTaftWilson}
\begin{split}
P_g^\varepsilon &=\k(1-g) \text{ for all }g\in\G,\text{ and if }\chi\neq\varepsilon, \text{ then}\\
P_g^\chi &\neq 0 \,\Longleftrightarrow\, g=g_i,\,\chi=\chi_i\text{  for some }1\le i\le\theta.
\end{split}
\end{align}
Thus we can choose $a_i\in P_{g_i}^{\chi_i}$ with residue class $x_i\in  V\#\k[\G]\cong A_1/A_0$ for $1\le i\le \theta$. 

\begin{lem}\label{LemQijCombinatorics2}
 Let  $w\in C(L)$ and $u\in D(L)$.
\begin{enumerate}
\item[\emph{(1)}] \emph{(a)} If $q_{i,w}\neq 1$ for some $1\le i\le\theta$, then $\chi_{w}\neq \varepsilon$.\\
           \emph{(b)} If $\chi_{w}\neq \varepsilon$ and for all $1\le i\le\theta$ there are $1\le j\le\theta$ such that $q_{j,w}\neq q_{ji}$ or $q_{w,j}\neq q_{ij}$,
then $$P_{g_{w}}^{\chi_{w}}=0.$$

\item[\emph{(2)}] Let $\ord q_{u,u}=N_u<\infty$.\\
           \emph{(a)} If $q_{i,u}^{N_u}\neq 1$ for some $1\le i\le\theta$, then $\chi_u^{N_u}\neq\varepsilon$.\\
          \emph{(b)} If $\chi_u^{N_u}\neq \varepsilon$ and for all $1\le i\le\theta$  there are $1\le j\le\theta$ such that
$q_{j,u}^{N_u}\neq q_{ji}$ or $q_{u,j}^{N_u}\neq q_{ij}$,
then $$P_{g_{u}^{N_u}}^{\chi_{u}^{N_u}}=0.$$
\end{enumerate}
\end{lem}
\begin{proof}
(1a) If $\chi_{w}=\varepsilon$, then $q_{i,w}=1$ for all $1\le i\le\theta$.\\
(1b) Let $\chi_{w}\neq \varepsilon$ and $P_{g_{w}}^{\chi_{w}}\neq 0$, then $\chi_{w}=\chi_i$ and $g_{w}=g_i$ for some $i$ by Eq.~\eqref{LemTaftWilson}. Hence $q_{j,w}=q_{ji}$ and $q_{w,j}=q_{ij}$ for all $1\le j\le\theta$.\\ 
(2a) If $\chi_u^{N_u}=\varepsilon$, then $q_{i,u}^{N_u}=1$ for all $1\le i\le\theta$.\\
(2b) Let $\chi_{u}^{N_u}\neq \varepsilon$ and $P_{g_{u}^{N_u}}^{\chi_{u}^{N_u}}\neq 0$, then $\chi_u^{N_u}=\chi_i$ and $g_{u}^{N_u}=g_i$ for some $i$. Thus $q_{j,u}^{N_u}=q_{ji}$ and $q_{u,j}^{N_u}=q_{ij}$ for all $1\le j\le\theta$. 
\end{proof}
This and Eq.~\eqref{LemTaftWilson} motivate the following:

\begin{defn}\label{DefnLiftCoeffic} Let $L\subset\Ly$.
Then we define coefficients
 $\mu_u\in\k$ for all $u\in D(L)$, and $\lambda_{w}\in \k$ for all $w\in C(L)$ by 
\begin{align*}
\mu_u=0,\text{ if } g_u^{N_u}=1\text{ or }\chi_u^{N_u}\neq \varepsilon, \qquad\quad
\lambda_{w}=0,\text{ if } g_{w}=1 \text{ or }\chi_{w}\neq \varepsilon,
\end{align*}
and otherwise they can be chosen arbitrarily.
\end{defn}

\subsection{General lifting procedure}\label{SectGeneralLiftProced} 

Suppose we know the PBW basis $[L]$ of $\BV$, then a lifting $A$ has the same PBW basis $[L]$; see \cite[Prop.~47]{UferPBW}. Hence we know by Section \ref{SectCharHA} the structure of the ideal $I$ such that 
$$
A\cong (\kX\# \k[\G])/I.
$$ 
Let us order the relations Eqs.~\eqref{ThmPBWCritIdluv} and \eqref{ThmPBWCritIdluN} of $I$, namely the two types
$$
[w]-\red{w}\text{ for $w\in C(L)$}\quad\text{and}\quad [u]^{N_u}-\redh{u}\text{ for }u\in D(L),
$$ 
with respect to $\prec$ by the leading super word $[w]$ resp.~$[u]^{N_u}$.
Yet we don't know the $\red{w},\redh{u}\in \kX\# \k[\G]$ explicitly; our general procedure to compute these elements is the following, stated inductively on $\prec$:

\begin{itemize}
 \item Suppose we know all relations $\prec$-smaller than $[w]$ resp.~$[u]^{N_u}$. 
\item Then we determine a counterterm $r_{w}$ resp.~$s_u\in \kX\# \k[\G]$ such that
$$
[w]-r_{w}\in P^{\chi_{w}}_{g_{w}}\quad\text{resp.}\quad  [u]^{N_u}-s_u\in P^{\chi_{u}^{N_u}}_{g_{u}^{N_u}}
$$
modulo the relations $\prec$-smaller than $[w]$ resp.~$[u]^{N_u}$; we conjecture that we can do this in general (see below).\\
Further if $\chi_{w}\neq \chi_i$ or $g_{w}\neq g_i$ resp.~$\chi_{u}^{N_u}\neq \chi_i$ or $g_{u}^{N_u} \neq g_i$ for all $1\le i\le\theta$, then by Eq.~\eqref{LemTaftWilson} we get
{\fboxsep=0.1in
\begin{align}\framebox{$
\red{w}=r_{w}+\lambda_{w}(1-g_{w})\quad\text{resp.}\quad \redh{u}=s_{u}+\mu_{u}(1-g_u^{N_u}).
$}
\end{align}
}
\end{itemize}

In order to formulate our conjecture, we define the following ideal: For any super word $U\in \SupW$ let $I_{U}$ denote the ideal of $\kX\#\k[\G]$ generated by the elements 
\begin{align*}
[w]         - \red{w}&    &&\text{for all } w\in C(L) \text{ and }[w]\prec U,\\
 [u]^{N_u}   - \redh{u}&    &&\text{for all }  u\in D(L) \text{ and }[u]^{N_u}\prec U.
\end{align*}
Note that $I_U\subset I$. 
\begin{conj}\label{ConjSkewPrim}
For all $w\in C(L)$ resp.~for all $u\in D(L)$ there are $r_{w}\in (\kX\# \k[\G])^{\chi_{w}}$ resp.~$s_u\in (\kX\# \k[\G])^{\chi_u^{N_u}}$ with  $r_{w}\prec_{L} [w]$ resp.~$s_u\prec_{L} [u]^{N_u}$  such that $[w]-r_{w}$ resp.~$[u]^{N_u}-s_u$ is skew-primitive modulo the relations $\prec$-smaller than $[w]$ resp.~$[u]^{N_u}$, i.e.,
\begin{align*}
\Delta([w]-r_{w})-([w]-r_{w})\o 1 -& g_{w}\o ([w]-r_{w}) \\ &\in  \kX\# \k[\G]\o I_{[w]} + I_{[w]}\o \kX\# \k[\G],\\
\Delta([u]^{N_u}-s_u)-([u]^{N_u}-s_u)\o 1 -& g_u^{N_u}\o ([u]^{N_u}-s_u)\\&\in  \kX\# \k[\G]\o I_{[u]^{N_u}} + I_{[u]^{N_u}}\o \kX\# \k[\G].
\end{align*}
\end{conj}

\begin{rem}\ 
\begin{enumerate}
\item  If the conjecture is true, then one could investigate from the list of braidings in \cite{HeckClassArtihRootSyst} where a free paramter $\lambda_{w}$ resp.~$\mu_u$ occurs in the lifting, without knowing $r_{w}$ resp.~$s_u$ explicitly.

\item To determine the generators of the ideal $I$ explicitly, i.e., to find $r_{w}$ resp.~$s_u$, 
it is crucial
to know which relations of $I$ are redundant. The redundant relations are detected as described in \cite{Helbig-PBW}.

\item In general $r_{w}$ resp.~$s_u$ is not necessarily in $\k[\G]$, like it was the case in \cite{AS-Class}; see Lemma \ref{LemLiftofGeneralRel} (2b),(3b) below or the liftings in the following sections.
\end{enumerate}
\end{rem}

At first we  lift the root vector relations of $x_1,\ldots,x_\theta$ and the Serre relations in general. Note that for these relations our Conjecture \ref{ConjSkewPrim} is true. We denote the images of $[x_i^rx_j],[x_ix_j^r]\in\kX$ ($r\ge 1$) of the algebra map $\kX\rightarrow A,\quad x_i\mapsto a_i$ by $[a_i^ra_j],[a_ia_j^r]$:
\begin{lem} \label{LemLiftofGeneralRel}
Let $A$ be a lifting of $\BV$ with braiding matrix $(q_{ij})$ and Cartan matrix $(a_{ij})$. Further let $1\le i<j\le\theta$ and $N_i:=\ord q_{ii}$. We may assume $q_{ii}\neq 1$ for all $1\le i\le\theta$. 
\begin{enumerate}
 \item[$(1)$] We have 
$$
a_i^{N_i}=\mu_i(1-g_i^{N_i}).
$$ Moreover, if $q_{ji}^{N_i}\neq 1$, then  $a_i^{N_i}=0$. Especially, if $q_{ii}=-1$ and $q_{ji}\neq\pm 1$, then $a_i^2=0$.
 \item[$(2)$] \emph{(a)} If $q_{ij}q_{ji}=q_{ii}^{a_{ij}}$, $N_i> 1-a_{ij}$, then 
$$
\bigl[a_i^{1-a_{ij}} a_j\bigr] =\lambda_{i^{1-a_{ij}}j} (1-g_{i}^{1-a_{ij}}g_j).
$$
Moreover,  if $
\begin{pmatrix}
 q_{ii} & q_{ij}\\
q_{ji} & q_{jj}
\end{pmatrix}
\neq 
\begin{pmatrix}
 q_{ii} & q_{ii}^{-(1-a_{ij})}\\
 q_{ii} & q_{ii}^{-(1-a_{ij})}
\end{pmatrix}
$, then $[a_i^{1-a_{ij}} a_j]=0$; in particular the latter claim holds  if $q_{jj}=-1$ and $q_{ii}^{1-a_{ij}}\neq -1$ (e.g., $N_i$ is odd).\\
\emph{(b)} If $N_i= 1-a_{ij} $, then $$\bigl[a_i^{N_i} a_j\bigr]=\mu_i(1-q_{ij}^{N_i})a_j.$$

 \item[$(3)$] \emph{(a)} If $q_{ij}q_{ji}=q_{jj}^{a_{ji}}$, $N_j> 1-a_{ji}$, then 
$$
\bigl[a_i a_j^{1-a_{ji}}\bigr]=\lambda_{ij^{1-a_{ji}}}(1-g_{i}g_j^{1-a_{ji}}).
$$
Moreover,  if $
\begin{pmatrix}
 q_{ii} & q_{ij}\\
q_{ji} & q_{jj}
\end{pmatrix}
\neq 
\begin{pmatrix}
 q_{jj}^{-(1-a_{ji})} & q_{jj}\\
 q_{jj}^{-(1-a_{ji})} & q_{jj}
\end{pmatrix}
$, then $[a_i a_j^{1-a_{ji}}]=0$; in particular the latter claim holds  if $q_{ii}=-1$ and $q_{jj}^{1-a_{ji}}\neq -1$ (e.g., $N_j$ is odd). \\ 
\emph{(b)} If  $N_j=1-a_{ji}$, then $$\bigl[a_i a_j^{N_j}\bigr]=\mu_j(q_{ji}^{N_j}-1)a_ig_j^{N_j}.$$
\end{enumerate}
\end{lem}
\begin{proof}
(1) This is a consequence of Lemma \ref{LemQijCombinatorics}(1)-(3) and Eq.~\eqref{LemTaftWilson}.

(2a) and (3a) follow from Lemma \ref{LemQijCombinatorics}(4)-(6) and Eq.~\eqref{LemTaftWilson}.

(2b) and (3b) follow from the restricted $q$-Leibniz formula of Proposition \ref{PropqCommut} and (1) above: For example
\begin{multline*}
 \bigl[a_i a_j^{N_j}\bigr]=\bigl[a_i, a_j^{N_j}\bigr]_{q_{ij}^{N_j}}=\bigl[a_i, \mu_j(1-g_j^{N_j})\bigr]_{q_{ij}^{N_j}}
 =\mu_j\Bigl((1-q_{ij}^{N_j})a_i -(1-q_{ij}^{N_j}q_{ji}^{N_j})a_ig_j^{N_j}\Bigr).
\end{multline*}
Now either $\mu_j=0$ or $q_{ij}^{N_j}=1$ by (1), from where the claim follows.
\end{proof}

From now on let $\theta =2$, i.e., $\BV$ is of rank 2.

\subsection{Lifting of $\BV$ with Cartan matrix $A_1\times A_1$}\label{LiftofNichAlgCartMatrA1xA1}
Let $\BV$ be a finite-dimensional Nichols algebras with Cartan matrix 
$(a_{ij})=\left(\begin{smallmatrix}  2 & 0\\ 0 & 2 \end{smallmatrix}\right)$ of type $A_1\times A_1$, i.e., the braiding matrix $(q_{ij})$ fulfills 
$$
q_{12}q_{21}=1,
$$
since we may suppose that $\ord q_{ii}\ge 2$ \cite[Sect.~2]{HeckWeylEquiv}, especially $q_{ii}\neq 1$. The Dynkin diagram is \rule[-3\unitlength]{0pt}{8\unitlength}
 \begin{picture}(14,5)(0,3)
 \put(1,2){\circle{2}}
 \put(13,2){\circle{2}}
 \put(1,5){\makebox[0pt]{\scriptsize $q$}}
 \put(13,5){\makebox[0pt]{\scriptsize $r$}}
 \end{picture} with $q:=q_{11}$ and $r:=q_{22}$. Then the Nichols algebra is given by
$$
\BV =T(V)/\bigl([x_1x_2],\, x_1^{N_1},\, x_2^{N_2}\bigr)
$$
with basis $\{x_2^{r_2} x_1^{r_1} \;|\;  0\le r_i< N_i \}$ where $N_i=\ord q_{ii}\ge 2$ \cite{HeckArtihRootSystRank2}.
It is well-known \cite{AS-p3} that any lifting $A$ is of the form
\begin{align*}
 A\cong (T(V)\#\k[\G])/\bigl(\quad [x_1x_2]-\lambda_{12} (1-g_{12}),\ x_1^{N_1}&-\mu_1(1-g_1^{N_1}),\\
                                  x_2^{N_2}&-\mu_2(1-g_2^{N_2})\quad\bigr)
\end{align*}
with basis $\{x_2^{r_2} x_1^{r_1} g \;|\;  0\le r_i< N_i,g\in\G \}$ and $\operatorname{dim}_\k A=N_1N_2\cdot |\G|$; the statement for the basis is proven in \cite{Helbig-PBW}.

\subsection{Lifting of $\BV$ with Cartan matrix $A_2$}\label{LiftofNichAlgCartMatrA2}
Let $\BV$ be a Nichols algebras with Cartan matrix 
$(a_{ij})=\left(\begin{smallmatrix}   2 & -1\\-1& 2    \end{smallmatrix}\right) $ 
of type $A_2$, i.e., the braiding matrix $(q_{ij})$ fulfills
$$
q_{12}q_{21}=q_{11}^{-1}\text{ or } q_{11}=-1,\text{ and } q_{12}q_{21}=q_{22}^{-1}\text{ or }q_{22}=-1.
$$
The Nichols algebras are given explicitly in \cite{HeckBV1}. As mentioned above, it is crucial to know the redundant relations for the computation of the liftings. Therefore we give the ideals without redundant relations which are detected in \cite{Helbig-PBW}:

\begin{prop}[Nichols algebras with Cartan matrix $A_2$]\label{ExNicholsAlgfinDimA2}  The finite-dimensional Nichols algebras $\BV$ with Cartan matrix of type $A_2$ are exactly the following:
\begin{enumerate}
\item[\emph{(1)}] \Dchaintwo{}{$q$}{$q^{-1}$}{$q$} (Cartan type $A_2$). Let $q_{12}q_{21}=q_{11}^{-1}=q_{22}^{-1}$ .\\
\emph{(a)} If $q_{11}=-1$, then 
$$ \BV=T(V)/ \bigl(x_1^{2},\ [x_1x_2]^{2},\ x_2^{2}\bigr) $$ 
with basis $\{x_2^{r_2}[x_1x_2]^{r_{12}} x_1^{r_1}\;|\; 0\le r_2,r_{12},r_1< 2\}$ and $\operatorname{dim}_\k \BV=2^3=8$.\\
\emph{(b)}
If $N:=\ord q_{11}\ge 3$, then 
$$ \qquad\qquad\BV=T(V)/ \bigl([x_1x_1x_2],\ [x_1x_2x_2],\ x_1^{N},\ [x_1x_2]^{N},\ x_2^{N}\bigr) $$ 
with basis $\{x_2^{r_2}[x_1x_2]^{r_{12}} x_1^{r_1}\;|\; 0\le r_2,r_{12},r_1< N\}$ and $\operatorname{dim}_\k \BV =N^3$.

\item[\emph{(2)}] \Dchaintwo{}{$q$}{$q^{-1}$}{$-1$}. If $q_{12}q_{21}=q_{11}^{-1}$, $N:=\ord q_{11}\ge 3$, $q_{22}=-1$, then 
$$  \BV=T(V)/ \bigl( [x_1x_1x_2],\  x_1^{N},\ x_2^2 \bigr)$$ 
with basis $\{x_2^{r_2}[x_1x_2]^{r_{12}} x_1^{r_1}\;|\;  0\le r_1< N,\,0 \le r_2,r_{12}< 2\}$ and $\operatorname{dim}_\k \BV=4N$.

\item[\emph{(3)}] \Dchaintwo{}{$-1$}{$q^{-1}$}{$q$}. If $q_{11}=-1$, $q_{12}q_{21}=q_{22}^{-1}$, $N:=\ord q_{22}\ge 3$, then
$$   \BV=T(V)/ \bigl([x_1x_2x_2],\ x_1^2,\ x_2^{N}\bigr)$$
with basis $\{x_2^{r_2}[x_1x_2]^{r_{12}} x_1^{r_1}\;|\;  0\le r_2< N,\,0 \le r_1,r_{12}< 2\}$ and $\operatorname{dim}_\k \BV=4N$.

\item[\emph{(4)}] \Dchaintwo{}{$-1$}{$q$}{$-1$}. If $q_{11}=q_{22}=-1$, $N:=\ord q_{12}q_{21}\ge 3$, then 
$$   \BV=T(V)/ \bigl( x_1^{2},\ [x_1x_2]^{N},\ x_2^{2}\bigr) $$ 
with basis $\{x_2^{r_2}[x_1x_2]^{r_{12}} x_1^{r_1}\;|\;  0\le r_2,r_1< 2,\,0 \le r_{12}< N\}$  and $\operatorname{dim}_\k \BV=4N$.
\end{enumerate}
\end{prop}

\begin{rem}
The Nichols algebras of Proposition \ref{ExNicholsAlgfinDimA2} all have the PBW basis $[L]=\{x_1,[x_1x_2],x_2\}$, and (1) resp. (2)-(4) form the standard Weyl equivalence class of row 2 resp.~3 in 
\cite[Figure 1]{HeckWeylEquiv,HeckArtihRootSystRank2}, where the latter is not of Cartan type. They build up the tree type $T_2$ of \cite{HeckBV1}.
\end{rem}

\begin{thm}[Liftings of $\BV$ with Cartan matrix $A_2$]\label{ExLiftingsA2} For any lifting $A$ of $\BV$ as in Proposition \ref{ExNicholsAlgfinDimA2}, we have
$$
A\cong (T(V)\# \k[\G])/ I,
$$ 
where $I$ is
specified as follows:
\begin{enumerate}
\item[\emph{(1)}] \Dchaintwo{}{$q$}{$q^{-1}$}{$q$} (Cartan type $A_2$). Let $q_{12}q_{21}=q_{11}^{-1}=q_{22}^{-1}$.\\
\emph{(a)}  If $q_{11}=-1$, then $I$ is generated by
\begin{align*}
x_1^{2}&-\mu_1(1-g_1^2),\\  
    [x_1x_2]^{2}& - 4\mu_1 q_{21}x_2^2-\mu_{12}(1-g_{12}^2),
    \\  
    x_2^{2}&-\mu_2(1-g_2^2).
\end{align*}
A basis is $\{x_2^{r_2}[x_1x_2]^{r_{12}} x_1^{r_1}g\;|\; 0\le r_2,r_{12},r_1< 2,\,g\in\G\}$   and $\operatorname{dim}_\k A=2^3\cdot|\G|=8\cdot|\G|$.\\
\emph{(b)} 
If $\ord q_{11}=3$, then $I$ is generated by, 
see \emph{\cite{Beattie}},
\begin{align*}
[x_1x_1x_2]-\lambda_{112}(1-g_{112}),\qquad\qquad\!  x_1^{3}&-\mu_1(1-g_1^3),\\
   [x_1x_2x_2]-\lambda_{122}(1-g_{122}),\qquad [x_1x_2]^{3}& +(1-q_{11})q_{11}\lambda_{112}[x_1x_2x_2] \\
         &-    \mu_1(1-q_{11})^3 x_2^3  -\mu_{12}(1-g_{12}^3),\\
x_2^{3}&-\mu_2(1-g_2^3).
\end{align*}
A basis is $\{x_2^{r_2}[x_1x_2]^{r_{12}} x_1^{r_1}g\;|\; 0\le r_2,r_{12},r_1< 3,\,g\in\G\}$ and $\operatorname{dim}_\k A=3^3\cdot|\G|=27\cdot|\G|$.\\
\emph{(c)} If $N:=\ord q_{11}\ge 4$, 
then  $I$ is generated by, see \emph{\cite{AS-A2}},
\begin{align*}
[x_1x_1x_2],\qquad\qquad\quad  x_1^{N}&-\mu_1(1-g_1^N),\\
[x_1x_2x_2],\qquad\quad   [x_1x_2]^{N}&-\mu_1(q_{11}-1)^N q_{21}^{\frac{N(N-1)}{2}} x_2^N-\mu_{12}(1-g_{12}^N),\\ 
              		 x_2^{N}&-\mu_2(1-g_2^N).
\end{align*}
A basis is $\{x_2^{r_2}[x_1x_2]^{r_{12}} x_1^{r_1}g\;|\; 0\le r_2,r_{12},r_1< N,\,g\in\G\}$ and $\operatorname{dim}_\k A=N^3\cdot|\G|$.

\item[\emph{(2)}] \Dchaintwo{}{$q$}{$q^{-1}$}{$-1$}. Let $q_{12}q_{21}=q_{11}^{-1}$, $q_{22}=-1$.\\
 \emph{(a)} If $4 \neq N:=\ord q_{11} \ge 3$, then $I$ is generated by 
\begin{align*}
[x_1x_1x_2], \qquad\quad x_1^{N}&-\mu_1(1-g_1^N),\\
  x_2^2  &-\mu_2(1-g_2^2).
\end{align*} A basis is $\{x_2^{r_2}[x_1x_2]^{r_{12}} x_1^{r_1}g\;|\;  0\le r_1< N,\,0 \le r_2,r_{12}< 2,\,g\in\G\}$  and $\operatorname{dim}_\k A=2^2N\cdot|\G|=4N\cdot|\G|$.\\
\emph{(b)} If $\ord q_{11}=4$, then $I$ is generated by
\begin{align*}
[x_1x_1x_2]-\lambda_{112}(1-g_{112}), \qquad\quad  x_1^{4}&-\mu_1(1-g_1^4),\\  x_2^2&-\mu_2(1-g_2^2).
\end{align*}  A basis is $\{x_2^{r_2}[x_1x_2]^{r_{12}} x_1^{r_1}g\;|\;  0\le r_1< 4,\,0 \le r_2,r_{12}< 2,\,g\in\G\}$  and $\operatorname{dim}_\k A=2^24\cdot|\G|=16\cdot|\G|$.

\item[\emph{(3)}] \Dchaintwo{}{$-1$}{$q^{-1}$}{$q$}. Let $q_{11}=-1$, $q_{12}q_{21}=q_{22}^{-1}$.\\
 \emph{(a)} If  $4 \neq N:= \ord q_{22}\ge 3$, then $I$ is generated by 
\begin{align*}
[x_1x_2x_2], \qquad\quad x_1^2  &-\mu_1(1-g_1^2) ,\\
  x_2^{N}&-\mu_2(1-g_2^N).
\end{align*} A basis is $\{x_2^{r_2}[x_1x_2]^{r_{12}} x_1^{r_1}g\;|\;  0\le r_2< N,\,0 \le r_1,r_{12}< 2,\,g\in\G\}$  and $\operatorname{dim}_\k A=2^2N\cdot|\G|=4N\cdot|\G|$.\\
\emph{(b)} If $\ord q_{22}=4$, then $I$ is generated by
\begin{align*}
[x_1x_2x_2]-\lambda_{122}(1-g_{122}), \qquad\quad  x_1^2&-\mu_1(1-g_1^2),\\  x_2^{4}&-\mu_2(1-g_2^4).
\end{align*} A basis is $\{x_2^{r_2}[x_1x_2]^{r_{12}} x_1^{r_1}g\;|\;  0\le r_2< 4,\,0 \le r_1,r_{12}< 2,\,g\in\G\}$  and $\operatorname{dim}_\k A=2^24\cdot|\G|=16\cdot|\G|$.

\item[\emph{(4)}]  \Dchaintwo{}{$-1$}{$q$}{$-1$}. Let $q_{11}=q_{22}=-1$ and $ N:=\ord q_{12}q_{21}\ge 3$.\\
\emph{(a)} If $q_{12}\neq \pm 1$, then $I$ is generated by
\begin{align*}
 x_1^{2}&-\mu_1(1-g_1^2),\\ 
[x_1x_2]^{N}&-\mu_{12}(1-g_{12}^N),\\
x_2^{2}&.\end{align*}
\emph{(b)} If $q_{12}=\pm 1$, then $I$ is generated by
\begin{align*} 
 x_1^{2}&,\\
[x_1x_2]^{N}&-\mu_{12}(1-g_{12}^N),\\
x_2^{2}&-\mu_2(1-g_2^2).
\end{align*}
In both cases a basis is $\{x_2^{r_2}[x_1x_2]^{r_{12}} x_1^{r_1}g\;|\;  0\le r_2,r_1< 2,\,0 \le r_{12}< N,\,g\in\G\}$  and $\operatorname{dim}_\k A=2^2N\cdot|\G|=4N\cdot|\G|$.

\end{enumerate}
\end{thm}
\begin{proof}
At first we show that in each case $(T(V)\# \k[\G])/ I$ is  a pointed Hopf algebra with coradical $\k[\G]$ and claimed basis and dimension  such that $\operatorname{gr}((T(V)\# \k[\G])/ I)\cong \BV\# \k[\G]$.
Then we show that a lifting $A$ is necessarily of this form.

\noindent \textbullet \ $(T(V)\# \k[\G])/ I$ is a Hopf algebra: We show that in every case $I$ is  generated by skew-primitive elements, thus $I$ is a Hopf ideal. The elements $x_i^{N_i}-\mu_i(1-g_i^{N_i})$ and $[x_1x_1x_2]-\lambda_{112}(1-g_{112})$ are skew-primitive if $q_{12}q_{21}=q_{11}^{-1}$ by Lemma \ref{LemSkewPrimRootSerre}. So we have a Hopf ideal in (2) and (3).

For the elements $[x_1x_2]^{N_{12}}-\redh{12}$ we argue as follows: In (1a) we directly calculate that $[x_1x_2]^2-4\mu_1q_{21}x_2^2\in P_{g_{12}^2}^{\chi_{12}^2}$. (1b),(1c) is treated in \cite{Beattie,AS-A2}.  

For (4a): By induction on $N$ (the induction basis $N=2$ is Lemma \ref{LemLiftofGeneralRel}(2b))
$$
[x_1(x_1x_2)^{N-1}]=\mu_1 \Bigl(\prod_{i=0}^{N-2}(1-q_{12}^{i+2}q_{21}^i)\Bigr) x_2 [x_1x_2]^{N-2}.
$$
Further $q_{12,12}=q_{12}q_{21}$ is of order $N$ and $q_{21}^2=1$ (or $\mu_1=0$), we have $q_{12}^N q_{21}^{N-2}=(q_{12}q_{21})^r=1$ and thus $[x_1(x_1x_2)^{N-1}]=0$. Hence $[x_1x_2]^N$ is skew-primitive by Lemma \ref{LemSkewZ}(1).

(4b) works in a similar way because of $q_{12}^2=1$: Again by induction (the induction basis $N=2$ is Lemma \ref{LemLiftofGeneralRel}(3b))
$$
[(x_1x_2)^{N-1}x_2] = \mu_2 \Bigl( \prod_{i=0}^{N-2}(1-q_{12}^{i+2}q_{21}^{i+2}) \Bigr)[x_1x_2]^{N-2}x_1g_2^2,
$$
which is 0 since $(q_{12}q_{21})^N=q_{12,12}^N=1$. Now $[x_1x_2]^N$ is skew-primitive by Lemma \ref{LemSkewZ}(2).

\noindent \textbullet \ The statement on the basis and dimension of $(T(V)\# \k[\G])/ I$ is proven in \cite{Helbig-PBW}.

\noindent \textbullet \ The  algebra \ $\k[\G]$ \ embeds  in \ $(T(V)\# \k[\G])/ I$ \ and the coradical of the latter is
 $((T(V)\# \k[\G])/ I)_0=\k[\G]$ 
\cite[Lem.~5.5.1]{Mont}, so $(T(V)\# \k[\G])/ I$ is pointed. 

\noindent \textbullet \ We consider the Hopf algebra map 
$$
T(V)\#\k[\G]\rightarrow \operatorname{gr}((T(V)\# \k[\G])/ I)
$$ 
which maps $x_i$ 
onto  the residue class of $x_i$  
in the homogeneous component of degree 1, namely 
$((T(V)\# \k[\G])/ I)_1/\k[\G]$. It is surjective, since $(T(V)\# \k[\G])/ I$ is generated as an algebra by $x_1$, $x_2$ and $\G$. Further it factorizes to 
$$
\BV\#\k[\G]\stackrel{\sim}{\rightarrow} \operatorname{gr}((T(V)\# \k[\G])/ I).
$$ 
This is a direct argument looking at the coradical filtration as in \cite[Cor.~5.3]{AS-p3}: all equations of $I$ are of the form $[w]-\red{w}$, $[u]^{N_u}-\redh{u}$ with $\red{w},\redh{u}\in \k[\G]=((T(V)\# \k[\G])/ I)_0$, hence $[w]=0$, $[u]^{N_u}=0$ in $ \operatorname{gr}((T(V)\# \k[\G])/ I)$. The latter surjective Hopf algebra map must be an isomorphism because the dimensions coincide.

\noindent \textbullet \ The other way round, let $A$ be a lifting of $\BV$ with $a_i\in P_{g_i}^{\chi_i}$ as in the beginning of this chapter. We consider the Hopf algebra map 
$$
T(V)\#\k[\G]\rightarrow A
$$ which takes $x_i$ to $a_i$ and $g$ to $g$. It is surjective since $A$ is generated by $a_1,a_2$ and $\G$ \cite[Lem.~2.2]{AS-p3}. We have to check whether this map factorizes to 
$$
(T(V)\#\k[\G])/I\stackrel{\sim}{\rightarrow} A.
$$ 
Then we are done since the dimension implies that this is an isomorphism.

But this means we have to check that the relations of $I$ hold in $A$: By Lemma \ref{LemLiftofGeneralRel} the relations concerning the elements $a_i^{N_i}$, $[a_1a_1a_2]$ and $[a_1a_2a_2]$ are of the right form. We are left to check those for $[a_1a_2]^{N_{12}}$, which appear in (1) and (4):

In (1a) we have  $[a_1a_2]^2-4\mu_1q_{21}a_2^2\in P_{g_{12}^2}^{\chi_{12}^2}$ like before. Now since $q_{1,12}^2=q_{12}^2\neq -1=q_{11}$ or $q_{12,2}^2=q_{12}^2 \neq q_{12}$, and $q_{1,12}^2= q_{12}^2\neq q_{12}$ or $q_{12,2}^2=q_{12}^2\neq -1= q_{22}$ (otherwise we get the contradiction $q_{12}=1$ and $q_{12}^2=-1$),  we have $[a_1a_2]^2=4\mu_1q_{21}a_2^2+\mu_{12}(1-g_{12}^2)$ by Lemma \ref{LemQijCombinatorics2}(2).
(1b),(1c) work in the same way; see \cite{Beattie,AS-A2}.
For (4): As shown before $[a_1a_2]^N\in P_{g_{12}^N}^{\chi_{12}^N}$. Again we deduce from Lemma \ref{LemQijCombinatorics2}(2) that $[a_1a_2]^N=\mu_{12}(1-g_{12}^N)$. 
\end{proof}

\begin{rem}
 The Conjecture \ref{ConjSkewPrim} is  true in the situation of Theorem \ref{ExLiftingsA2}: the $r_{w}$ of the non-redundant relations $[w]-\red{w}$ are 0 ($r_{112}=r_{122}=0$ if the Serre relations are not redundant)
 and $s_{12}\in \k[\G]$ in (1), otherwise $s_u=0$ if $[u]^{N_u}-\redh{u}$ is not redundant.
\end{rem}

\subsection{Lifting of $\BV$ with Cartan matrix $B_2$}\label{LiftofNichAlgCartMatrB2}

In this section we lift some of the Nichols algebras of standard type with associated Cartan matrix $\left(\begin{smallmatrix}2&-3\\-1&2 \end{smallmatrix}\right)$ of type $B_2$ (in the next Section also of non-standard type $B_2$). At first we recall the Nichols algebras (see \cite{HeckBV1}), but again we give  the ideals without redundant relations \cite{Helbig-PBW}:

\begin{prop}[Nichols algebras with Cartan matrix $B_2$]\label{ExNicholsAlgfinDimB2}  The following finite-dimensional Nichols algebras $\BV$ of standard type with braiding matrix $(q_{ij})$ and Cartan matrix of type $B_2$ are represented as follows:
\begin{enumerate}
\item[\emph{(1)}] \Dchaintwo{}{$q$}{$q^{-2}$}{$q^2$} (Cartan type $B_2$). Let $q_{12}q_{21}=q_{11}^{-2}=q_{22}^{-1}$ and $N:=\ord q_{11}$.\\
\emph{(a)} If $N= 3$, then
$$
\BV=T(V)/\bigl( [x_1x_2x_2]  , \  x_1^{3} ,\     [x_1x_1x_2]^{3},\    [x_1x_2]^{3}  , \   x_2^{3}\bigr)
$$
with basis
$\bigl\{x_2^{r_2}[x_1x_2]^{r_{12}}[x_1x_1x_2]^{r_{112}} x_1^{r_1}\ |\  0\le r_1,r_{12},r_{112}, r_2<3 \bigr\}$ and $\operatorname{dim}_\k\BV=3^4=81$.\\
\emph{(b)} If $N=4$, then
$$
\BV=T(V)/\bigl( [x_1x_1x_1x_2]  , \  x_1^{4} ,\     [x_1x_1x_2]^{2},\    [x_1x_2]^{4}  , \   x_2^{2}\bigr)
$$
with basis 
$\bigl\{x_2^{r_2}[x_1x_2]^{r_{12}}[x_1x_1x_2]^{r_{112}} x_1^{r_1}\ |\  0\le r_1,r_{12}< 4,\ 0 \le r_2,r_{112}< 2\bigr\}$  and $\operatorname{dim}_\k\BV=2^2\cdot 4^2=64$.\\
\emph{(c)} If $N\ge 5$ is odd, then
$$
\BV=T(V)/\bigl(  [x_1x_1x_1x_2] , \ [x_1x_2x_2]  , \  x_1^{N} ,\     [x_1x_1x_2]^{N},\    [x_1x_2]^{N}  , \   x_2^{N} \bigr)
$$
with basis
$\bigl\{x_2^{r_2}[x_1x_2]^{r_{12}}[x_1x_1x_2]^{r_{112}} x_1^{r_1}\ |\  0\le r_1,r_{12},r_{112}, r_2<N \bigr\}$  and $\operatorname{dim}_\k\BV=N^4$.\\
\emph{(d)} If $N\ge 6$ is even, then
$$
\BV=T(V)/\bigl(  [x_1x_1x_1x_2] , \ [x_1x_2x_2]  , \  x_1^{N} ,\     [x_1x_1x_2]^{\frac{N}{2}},\    [x_1x_2]^{N}  , \   x_2^{\frac{N}{2}} \bigr)
$$
with basis 
$\bigl\{x_2^{r_2}[x_1x_2]^{r_{12}}[x_1x_1x_2]^{r_{112}} x_1^{r_1}\ |\  0\le r_1,r_{12}< N,\ 0 \le r_2,r_{112}<\frac{N}{2} \bigr\}$ and $\operatorname{dim}_\k\BV=\frac{N^4}{4}$.

\item[\emph{(2)}]  \Dchaintwo{}{$q$}{$q^{-2}$}{$-1$}, \Dchaintwo{}{$-q^{-1}$}{$q^{2}$}{$-1$}.  Let $q_{12}q_{21}=q_{11}^{-2}$, $q_{22}=-1$ and $N:=\ord q_{11}$.\\
\emph{(a)} If $N=3$,  then 
\begin{align*}
\BV=T(V)/\bigl( [x_1x_1x_2x_1x_2],\   x_1^{3},\    [x_1x_2]^{6}, \    x_2^{2}\bigr)
\end{align*}
with basis 
$
\bigl\{x_2^{r_2}[x_1x_2]^{r_{12}}[x_1x_1x_2]^{r_{112}} x_1^{r_1}\ |\  0\le r_1<3,\ 0\le r_{12}< 6,\ 0 \le r_2,r_{112}< 2\ \bigr\}
$ and $\operatorname{dim}_\k \BV=72.$\\
\emph{(b)} If $N\ge 5$ ($N=4$ is \emph{(1b)}), then for $N':=\ord(-q_{11}^{-1})$
\begin{align*}
\BV=T(V)/\bigl([x_1x_1x_1x_2],\   x_1^{N},\   [x_1x_2]^{N'}, \  x_2^{2} \bigr)
\end{align*} 
with basis $
\bigl\{x_2^{r_2}[x_1x_2]^{r_{12}}[x_1x_1x_2]^{r_{112}} x_1^{r_1}\ |\  0\le r_1<N,\ 0\le r_{12}< N',\ 0 \le r_2,r_{112}< 2 \bigr\}$ and $\operatorname{dim}_\k\BV=4NN'.$

\item[\emph{(3)}] \Dchaintwo{}{$\zeta$}{$q^{-1}$}{$q$}, \Dchaintwo{}{$\zeta$}{$\zeta^{-1}q$}{$\ \ \zeta q^{-1}$}.  Let $\ord q_{11}=3$, $q_{12}q_{21}=q_{22}^{-1}$ and $N:=\ord q_{22}$. \\
\emph{(a)} If $N=2$,  then
\begin{align*}
\BV=T(V)/\bigl(  [x_1x_1x_2x_1x_2],\  x_1^{3},\ [x_1x_1x_2]^6,\     x_2^{2}\bigr)
\end{align*}
with basis 
$
\bigl\{x_2^{r_2}[x_1x_2]^{r_{12}} [x_1x_1x_2]^{r_{112}}x_1^{r_1}\ |\  0\le r_1,r_{12}<3,\ 0 \le r_2< 2,\ 0 \le r_{112}< 6 \bigr\}
$
and $\operatorname{dim}_\k\BV=108 .$

\emph{(b)} If $N\ge 4$ ($N=3$ is \emph{(1)} or  Proposition \ref{ExNicholsAlgfinDimA2}\emph{(1)}), then for $N':=\ord q_{11}q_{22}^{-1}$
$$
\BV=T(V)/\bigl(  [x_1x_2x_2],\  x_1^{3},\ [x_1x_1x_2]^{N'},\     x_2^{N}\bigr)
$$
with basis 
$
\bigl\{x_2^{r_2}[x_1x_2]^{r_{12}} [x_1x_1x_2]^{r_{112}}x_1^{r_1}\ |\  0\le r_1,r_{12}<3,\ 0 \le r_2< N,\ 0 \le r_{112}< N' \bigr\}
$
and $\operatorname{dim}_\k\BV=9 NN' .$

\item[\emph{(4)}] \Dchaintwo{}{$\zeta$}{$-\zeta$}{$-1$}, \Dchaintwo{}{$\zeta^{-1}$}{$-\zeta^{-1}$}{$-1$}. Let $\ord q_{11}=3$, $q_{12}q_{21}=-q_{11}$, $q_{22}=-1$, then
\begin{align*}
\BV=T(V)/\bigl( [x_1x_1x_2x_1x_2],\   x_1^{3},\       x_2^{2}  \bigr)
\end{align*} 
with basis 
$
\bigl\{x_2^{r_2}[x_1x_2]^{r_{12}}[x_1x_1x_2]^{r_{112}} x_1^{r_1}\ |\  0\le r_1,r_{12}<3,\ 0 \le r_2,r_{112}< 2 \bigr\}
$
and $\operatorname{dim}_\k A=36.$
\end{enumerate}
\end{prop}

\begin{rem}
The Nichols algebras of Proposition \ref{ExNicholsAlgfinDimB2} all have the PBW basis $[L]=\{x_2, [x_1x_2], [x_1x_1x_2], x_1\}$, and (1)-(4) form the standard Weyl equivalence classes of row 4-7 in 
\cite[Figure 1]{HeckWeylEquiv,HeckArtihRootSystRank2}, where the rows 5-7 are not of Cartan type. They build up the tree type $T_3$ of \cite{HeckBV1}.
\end{rem}

\begin{thm}[Liftings of $\BV$ with Cartan matrix $B_2$]\label{ExLiftingsB2}
For any lifting $A$ of $\BV$ as in Proposition \ref{ExNicholsAlgfinDimB2}, we have
$$
A\cong (T(V)\# \k[\G])/ I,
$$ 
where $I$ is 
specified as follows:
\begin{enumerate}
\item[\emph{(1)}] \Dchaintwo{}{$q$}{$q^{-2}$}{$q^2$} (Cartan type $B_2$). Let $q_{12}q_{21}=q_{11}^{-2}=q_{22}^{-1}$ .\\
\emph{(a)} If $\ord q_{11}=4$ and 
  $q_{12}\neq \pm 1$, then $I$ is generated by
\begin{align*}
[x_1x_1x_1x_2]  , \qquad\qquad\qquad
 x_1^{4}&-\mu_1(1-g_1^4),\\
         [x_1x_1x_2]^{2}&,\\
      [x_1x_2]^{4}&-\mu_{12}(1-g_{12}^4)  , \\ 
                  x_2^{2}&.
\end{align*}
\emph{(b)} If $\ord q_{11}=4$ and $q_{12}=\pm 1$, then $I$ is generated by
\begin{align*}
[x_1x_1x_1x_2], \qquad\qquad\qquad
  x_1^{4}&-\mu_1(1-g_1^4),\\
         [x_1x_1x_2]^{2}& -8q_{11}\mu_1 x_2^2 -\mu_{112}(1-g_{112}^2),\\
         [x_1x_2]^{4}&-16\mu_1x_2^4+4\mu_{112}q_{11}x_2^2-\mu_{12}(1-g_{12}^4)  , \\ 
                   x_2^{2}&-\mu_2(1-g_2^2).
\end{align*}
In both \emph{(a)} and \emph{(b)} $\operatorname{dim}_\k A=2^24^2\cdot |\G|=128 \cdot |\G|$ and a basis is 
$$
\bigl\{x_2^{r_2}[x_1x_2]^{r_{12}}[x_1x_1x_2]^{r_{112}} x_1^{r_1}g\ |\  0\le r_1,r_{12}< 4,\ 0 \le r_2,r_{112}< 2,\ g\in\G\bigr\}.
$$ 

\item[\emph{(2)}] \Dchaintwo{}{$q$}{$q^{-2}$}{$-1$}, \Dchaintwo{}{$-q^{-1}$}{$q^{2}$}{$-1$}.  Let $q_{12}q_{21}=q_{11}^{-2}$, $q_{22}=-1$.\\
\emph{(a)} If $\ord q_{11}=3$ and $q_{12}\neq \pm 1$,  then $I$ is generated by
\begin{align*}
[x_1x_1x_2x_1x_2], \qquad\qquad
            x_1^{3}&-\mu_1(1-g_1^3),\\
     [x_1x_2]^{6}&-\mu_{12}(1-g_{12}^6), \\ 
    x_2^{2}&.
\end{align*}
\emph{(b)} If $\ord q_{11}=3$ and $q_{12}=-1$, then $I$ is generated by
\begin{align*}
[x_1x_1x_2x_1x_2], \qquad\qquad
            x_1^{3}&,\\
     [x_1x_2]^{6} &-\mu_{12}(1-g_{12}^6), \\ 
  x_2^{2}&-\mu_2(1-g_2^2).
\end{align*}
\emph{(c)} If $\ord q_{11}=3$ and $q_{12}=  1$,  then $I$ is generated by
\begin{align*}
[x_1x_1x_2x_1x_2]&+3\mu_1(1-q_{11})x_2^2 -\lambda_{11212}(1-g_{11212}),\\
            x_1^{3}&-\mu_1(1-g_1^3),\\
     [x_1x_2]^{6}&-s_{12}         -\mu_{12}(1-g_{12}^6), \\ 
  x_2^{2}&-\mu_2(1-g_2^2),
\end{align*}
where
\begin{align*}
 s_{12}:=-3\mu_2 \Bigl\{&(\lambda_{11212}( 1-q_{11})+9\mu_1\mu_2q_{11})[x_1x_2]^2 x_1g_2^2\\
         &-q_{11}(\lambda_{11212}(1-q_{11})+9\mu_1\mu_2q_{11})[x_1x_2][x_1x_1x_2]g_2^2\\
         &+(\lambda_{11212}^2q_{11}^2+3\mu_1\mu_2\lambda_{11212}(1-q_{11}^2)-9\mu_1^2\mu_2^2) g_1^6g_2^6\\
         &+ 3\mu_1\mu_2(\lambda_{11212}(1-q_{11}^2)-3\mu_1\mu_2)g_1^3g_2^6\\
         &+ \lambda_{11212} (3\mu_1\mu_2(q_{11}-1)+\lambda_{11212})g_1^3g_2^4\\
         &-9\mu_1^2\mu_2^2g_2^6\\
         &+ 3\mu_1\mu_2 (\lambda_{11212}(q_{11}-1)-9\mu_1\mu_2q_{11})g_2^4\\
         &+ q_{11}( \lambda_{11212}^2 -6\mu_1\mu_2\lambda_{11212} (1-q_{11}) -27\mu_1^2\mu_2^2q_{11} )g_2^2\Bigr\}.
\end{align*}
In \emph{(a),(b),(c)}  $\operatorname{dim}_\k A=72 \cdot |\G|.$ and a basis is 
\begin{align*}
 \bigl\{x_2^{r_2}[x_1x_2]^{r_{12}}[x_1x_1x_2]^{r_{112}} x_1^{r_1}g\ |\  0\le r_1<3,\ 0\le r_{12}< 6,\ 0 \le r_2,r_{112}< 2,\ g\in\G\bigr\}.
\end{align*}
\emph{(d)} Let $N:=\ord q_{11}>4$ ($N=4$ is \emph{(1)}), and $q_{12}\neq \pm 1$. Denote $$N':=\ord(-q_{11}^{-1})=\begin{cases}
     2N, & \text{if }N\text{ odd},\\
N/2 ,& \text{if }N\text{ even and }N/2\text{ odd},\\
N, & \text{if }N,N/2\text{ even}.
    \end{cases}
 $$ 
Then $I$ is generated by
\begin{align*}
[x_1x_1x_1x_2], \qquad\qquad\qquad
            x_1^{N}&-\mu_1(1-g_1^N),\\
     [x_1x_2]^{N'} &-\mu_{12}(1-g_{12}^{N'}), \\ 
  x_2^{2}&.
\end{align*} It is $\operatorname{dim}_\k A=4NN' \cdot |\G|$ and a basis is 
$$
\bigl\{x_2^{r_2}[x_1x_2]^{r_{12}}[x_1x_1x_2]^{r_{112}} x_1^{r_1}g\ |\  0\le r_1<N,\ 0\le r_{12}< N',\ 0 \le r_2,r_{112}< 2,\ g\in\G\bigr\}.
$$ 

\item[\emph{(3)}]  \Dchaintwo{}{$\zeta$}{$q^{-1}$}{$q$},  \Dchaintwo{}{$\zeta$}{$\zeta^{-1}q$}{\ \ $\zeta q^{-1}$}. Let $\ord q_{11}=3$, $q_{12}q_{21}=q_{22}^{-1}$. \\
\emph{(a)} If $q_{22}=-1$ and $q_{12}\neq \pm 1$,  then $I$ is generated by
\begin{align*}
[x_1x_1x_2x_1x_2], \qquad\qquad\qquad
  x_1^{3}&-\mu_1(1-g_1^3),\\
[x_1x_1x_2]^6&-\mu_{112}(1-g_{112}^6)\\
      x_2^{2}&.
\end{align*}
\emph{(b)} If $q_{22}=-1$ and $q_{12}= 1$,  then $I$ is generated by
\begin{align*}
[x_1x_1x_2x_1x_2], \qquad\qquad\qquad
  x_1^{3}&,\\
[x_1x_1x_2]^6&-\mu_{112}(1-g_{112}^6)\\
      x_2^{2}&-\mu_2(1-g_2^2).
\end{align*}
\emph{(c)} If $q_{22}=-1$ and $q_{12}=-1$,  then $I$ is generated by
\begin{align*}
[x_1x_1x_2x_1x_2]&+4\mu_2 x_1^3 g_2^2  -\lambda_{11212} (1-g_1^3 g_2^2),\\
  x_1^{3}&-\mu_1(1-g_1^3),\\
[x_1x_1x_2]^6 &- s_{112}-\mu_{112}(1-g_{112}^6)\\
      x_2^{2}&-\mu_2(1-g_2^2),
\end{align*}
where
\begin{align*}
 s_{112}:=-2\mu_1\Bigl\{& 2(- \lambda_{11212}+4 \mu_1\mu_2) q_{11}(1-q_{11})   x_2 [x_1x_1x_2]^3 g_1^3 g_2^2\\
&+2 (\lambda_{11212}-4\mu_1\mu_2)  q_{11}(1-q_{11}) [x_1x_2]^2 [x_1x_1x_2]^2 g_1^3 g_2^2\\
&+2(\lambda_{11212}^2-8\mu_1\mu_2\lambda_{11212}+16\mu_1^2\mu_2^2)q_{11}(1-q_{11})  [x_1x_2] [x_1x_1x_2] g_1^6 g_2^4\\
&+8\mu_1\mu_2(\lambda_{11212}-4 \mu_1\mu_2)q_{11}(1-q_{11}) [x_1x_2] [x_1x_1x_2] g_1^3 g_2^4\\
&+2\lambda_{11212}(-\lambda_{11212}+4\mu_1\mu_2) q_{11}(1-q_{11})[x_1x_2] [x_1x_1x_2] g_1^3 g_2^2\\
&+2(-\lambda_{11212}^3+6 \mu_1\mu_2\lambda_{11212}^2-16\mu_1^2\mu_2^2 \lambda_{11212}+16\mu_1^3\mu_2^3)         g_1^{12} g_2^6\\
&+(- \lambda_{11212}^3+12 \mu_1 \mu_2 \lambda_{11212}^2-48 \mu_1^2\mu_2^2\lambda_{11212}+64 \mu_1^3\mu_2^3)q_{11}(1-q_{11})  g_1^9 g_2^6\\
&+10 \mu_1 \mu_2(- \lambda_{11212}^2+8 \mu_1 \mu_2 \lambda_{11212}-16 \mu_1^2 \mu_2^2)        g_1^6 g_2^6\\
&+2( \lambda_{11212}^3-7 \mu_1 \mu_2 \lambda_{11212}^2+8 \mu_1^2\mu_2^2\lambda_{11212}+16 \mu_1^3 \mu_2^3)          g_1^6 g_2^4\\
&+16\mu_1^2\mu_2^2(   \lambda_{11212}-4 \mu_1 \mu_2)  q_{11}(1-q_{11})   g_1^3 g_2^6\\
&+8\mu_1 \mu_2\lambda_{11212}(-  \lambda_{11212}+4 \mu_1 \mu_2 ) q_{11}(1-q_{11})  g_1^3 g_2^4\\
&+32 \mu_1^3 \mu_2^3      g_2^6\end{align*}\begin{align*}
&+ \lambda_{11212}^2(  \lambda_{11212}-4 \mu_1 \mu_2) q_{11}(1-q_{11})    g_1^3 g_2^2\\
&+32\mu_1^2 \mu_2^2(- \lambda_{11212}+ \mu_1 \mu_2)       g_2^4\\
&+4\mu_1\mu_2(3  \lambda_{11212}^2-8 \mu_1 \mu_2 \lambda_{11212}+8 \mu_1^2 \mu_2^2)      g_2^2\Bigr\}
\end{align*}
In \emph{(a),(b),(c)}  $\operatorname{dim}_\k A=108 \cdot |\G|$ and a basis is 
$$
\bigl\{x_2^{r_2}[x_1x_2]^{r_{12}} [x_1x_1x_2]^{r_{112}}x_1^{r_1}g\ |\  0\le r_1,r_{12}<3,\ 0 \le r_2< 2,\ 0 \le r_{112}< 6,\ g\in\G\bigr\}.
$$ 

\item[\emph{(4)}] \Dchaintwo{}{$\zeta$}{$-\zeta$}{$-1$},  \Dchaintwo{}{$\zeta^{-1}$}{$-\zeta^{-1}$}{$-1$}.  Let $\ord q_{11}=3$, $q_{12}q_{21}=-q_{11}$ of order $6$, $q_{22}=-1$.\\
\emph{(a)} If $q_{12}\neq \pm 1$,  then $I$ is generated by 
\begin{align*}
[x_1x_1x_2x_1x_2], \qquad\qquad\qquad
  x_1^{3}&-\mu_1(1-g_1^3),\\
      x_2^{2}&.
\end{align*}
\emph{(b)} If $q_{12}= 1$,  then $I$ is generated by
\begin{align*}
[x_1x_1x_2x_1x_2], \qquad\qquad\qquad
  x_1^{3}&,\\
      x_2^{2}&-\mu_2(1-g_2^2).
\end{align*}
\emph{(c)} If $q_{12}= -1$,  then $I$ is generated by
\begin{align*}
[x_1x_1x_2x_1x_2]&-
\mu_2(1+q_{11})x_1^3g_2^2 -\lambda_{11212}(1-g_{11212}),\\
  x_1^{3}&-\mu_1(1-g_1^3),\\
      x_2^{2}&-\mu_2(1-g_2^2).
\end{align*}
In \emph{(a),(b),(c)} $\operatorname{dim}_\k A=36 \cdot |\G|$ and a basis is 
$$
\bigl\{x_2^{r_2}[x_1x_2]^{r_{12}}[x_1x_1x_2]^{r_{112}} x_1^{r_1}g\ |\  0\le r_1,r_{12}<3,\ 0 \le r_2,r_{112}< 2,\ g\in\G\bigr\}.
$$ 
\end{enumerate}
\end{thm}
\begin{proof}
We proceed as in the proof of Theorem \ref{ExLiftingsA2}. 

\noindent \textbullet \ $(T(V)\# \k[\G])/ I$ is a Hopf algebra, since $I$ is generated by skew-primitive elements:
 Again the elements $x_i^{N_i}-\mu_i(1-g_i^{N_i})$ and $[x_1x_1x_1x_2]-\lambda_{1112}(1-g_{1112})$ are skew-primitive if $q_{12}q_{21}=q_{11}^{-2}$ by Lemma \ref{LemSkewPrimRootSerre}. 

(1a) By Lemma \ref{LemSkewZ}(1)  $[x_1x_2]^{4}\in P_{g_{12}^4}^{\chi_{12}^4}$ and hence also  $[x_1x_2]^{4}-\mu_{12}(1-g_{12}^4) \in P_{g_{12}^4}^{\chi_{12}^4}$. A direct computation yields $[x_1x_1x_2]^{2}\in P_{g_{112}^2}^{\chi_{112}^2}$.

(1b) Again direct computation shows that  $[x_1x_1x_2]^{2} -8q_{11}\mu_1 x_2^2 -\mu_{112}(1-g_{112}^2)$ and 
$[x_1x_2]^{4}-16\mu_1x_2^4+4\mu_{112}q_{11}x_2^2-\mu_{12}(1-g_{12}^4)$ are skew primitive; we used the computer algebra system FELIX \cite{felix}.

(2a) We have $[x_1x_1x_2x_1x_2]\in P_{g_{11212}}^{\chi_{11212}}$ by Lemma \ref{LemSkew11212}(2). Further  $[x_1x_2]^{6}\in P_{g_{12}^6}^{\chi_{12}^6}$ by Lemma \ref{LemSkewZ}(1).

(2b) Again $[x_1x_1x_2x_1x_2]\in P_{g_{11212}}^{\chi_{11212}}$ by Lemma \ref{LemSkew11212}(2) and a direct computation yields $[x_1x_2]^{6}-\mu_{12}(1-g_{12}^6)\in P_{g_{12}^6}^{\chi_{12}^6}$.

(2c) Using FELIX we get that all elements are skew-primitive.

(2d) This is again Lemma \ref{LemSkewZ}(1).

(3a) and (3b):   $[x_1x_1x_2x_1x_2]\in P_{g_{11212}}^{\chi_{11212}}$ by Lemma \ref{LemSkew11212}(1). Straightforward calculation shows that $[x_1x_1x_2]^{6} -\mu_{112}(1-g_{112}^6)\in P_{g_{112}^6}^{\chi_{112}^6}$; here again we used FELIX.

(3c) is computed using FELIX.

(4a) and (4b): $[x_1x_1x_2x_1x_2]\in P_{g_{11212}}^{\chi_{11212}}$ by Lemma \ref{LemSkew11212}(1).

(4c) Looking at the coproduct computed in Lemma \ref{LemSkew11212}(1) we deduce that the element $[x_1x_1x_2x_1x_2]- 
\mu_2(1+q_{11})x_1^3g_2^2$ and hence $[x_1x_1x_2x_1x_2]- \mu_2(1+q_{11})x_1^3g_2^2 -\lambda_{11212}(1-g_{11212})$ is skew-primitive.

\noindent \textbullet \ The statement on the basis and dimension of $(T(V)\# \k[\G])/ I$ is proven  in \cite{Helbig-PBW}.

\noindent \textbullet \ $(T(V)\# \k[\G])/ I$ is pointed by the same argument as in the proof of Theorem \ref{ExLiftingsA2}.

\noindent \textbullet \ The surjective Hopf algebra map as given  in the proof of Theorem \ref{ExLiftingsA2}
$$
T(V)\#\k[\G]\rightarrow  \operatorname{gr}((T(V)\# \k[\G])/ I)
$$ 
factorizes to an isomorphism $
\BV\#\k[\G]\stackrel{\sim}{\rightarrow} \operatorname{gr}((T(V)\# \k[\G])/ I):
$
Again we look at the coradical filtration. All equations of $I$ are of the form $[w]-\red{w}$, $[u]^{N_u}-\redh{u}$ with $\red{w}$ resp.~$\redh{u}$ of lower degree in $ \operatorname{gr}((T(V)\# \k[\G])/ I)$, hence $[w]=0$, $[u]^{N_u}=0$ in $ \operatorname{gr}((T(V)\# \k[\G])/ I)$.

\noindent \textbullet \ Like before, for a lifting $A$ we have to check whether the surjective Hopf algebra map 
$$
T(V)\#\k[\G]\rightarrow A
$$ which takes $x_i$ to $a_i$ and $g$ to $g$  factorizes to
$$
(T(V)\#\k[\G])/I\stackrel{\sim}{\rightarrow} A.
$$

By Lemma \ref{LemLiftofGeneralRel} the relations concerning the elements $a_i^{N_i}$ and $[a_1a_1a_1a_2]$ are of the right form. We deduce from Lemma \ref{LemQijCombinatorics2} that the relations also hold in $A$: this is just combinatorics on the braiding matrices which we want to demonstrate for the following.

(1a) We have $\chi_{112}^2\neq \varepsilon$ by Lemma \ref{LemQijCombinatorics2}(2a), since $q_{1,112}^2=q_{12}^2\neq 1$.
Further $P_{g_{112}^2}^{\chi_{112}^2}=0$ by Lemma \ref{LemQijCombinatorics2}(2b): Suppose $q_{21}^4q_{22}^2=q_{21}$, $q_{12}^4q_{22}^2=q_{12}$, then $q_{21}^3=q_{12}^3=1$, which contradicts $q_{12}q_{21}=q_{11}^{-2}=-1$; also if $q_{11}^4q_{12}^2=q_{12}$, $q_{11}^4q_{21}^2=q_{21}$, then $q_{12}=q_{21}=1$, again a contradiction to $q_{12}q_{21}=q_{11}^{-2}=-1$. Hence $[a_1a_1a_2]^{2}=0$. 
The other cases work in exactly the same manner.
\end{proof}

\begin{rem}
The Conjecture \ref{ConjSkewPrim} is true in the above cases. Further note that in (2c) $s_{12}\notin \k[\G]$ and in (3c) $s_{112}\notin \k[\G]$.

Further we want to note the cases not treated in the theorem above:
\begin{enumerate}
 \item The case (1) when $5\neq N:=\ord q_{11}\ge 3$ is odd is treated in \cite{Beattie}, and the case $N=5$ in \cite{Didt}.

\item There is no general method for  (1) in the case $N:=\ord q_{11}\ge 6$  is even. Here $\ord q_{22}=\ord q_{11}^2=\frac{N}{2}$.

\item There is no general method for (2d) in the case $q_{12}= \pm 1$.

\item There is no general method for (3) in the case  $N:=\ord q_{22}\ge 4$. The case $N=3$ is (1) of the theorem above or (1) of Theorem \ref{ExLiftingsA2}.
\end{enumerate}

\end{rem}

\subsection{Lifting of $\BV$ of non-standard type}\label{LiftofNichAlgNonStandardType}

In this section we want to lift some of the Nichols algebras of the Weyl equivalence classes of rows 8 and 9 of \cite[Figure 1]{HeckWeylEquiv,HeckArtihRootSystRank2} which are not of standard type,  namely for $\ord\zeta =12$ we lift  
 \begin{center}
  \Dchaintwo{}{$-\zeta ^{-2}$}{$-\zeta ^3$}{$-\zeta ^2$},  \
 \Dchaintwo{}{$-\zeta ^{-2}$}{$\zeta ^{-1}$}{$-1$},\
 \Dchaintwo{}{$-\zeta ^2$}{$-\zeta $}{$-1$},\
 \Dchaintwo{}{$-\zeta ^3$}{$\zeta $}{$-1$},\
\Dchaintwo{}{$-\zeta ^3$}{$-\zeta ^{-1}$}{$-1$}
\end{center}
 of row 8, and
\begin{center}
 \Dchaintwo{}{$-\zeta ^2$}{$\zeta ^3$}{$-1$},\
 \Dchaintwo{}{$-\zeta ^{-1}$}{$-\zeta ^3$}{$-1$}
\end{center}
of row 9. Again, at first we give a nice presentation of the ideal cancelling the redundant relations of the ideals given in \cite{HeckBV1}:

\begin{prop}[Nichols algebras of rows 8 and 9]\label{PropNicholsAlgRows89} The following finite-dimensional Nichols algebras $\BV$  with braiding matrix $(q_{ij})$ of rows 8 and 9 of \cite[Figure 1]{HeckWeylEquiv,HeckArtihRootSystRank2} are represented as follows:
 Let $\zeta\in\k^\times$, $\ord\zeta =12$.
\begin{enumerate}
 \item[\emph{(1)}] 
 \Dchaintwo{}{$-\zeta ^{-2}$}{$-\zeta ^3$}{$-\zeta ^2$}. Let $q_{11}=-\zeta^{-2}$,  
$
q_{12}q_{21}=-\zeta^{3}
$, $q_{22}=-\zeta^{2}$, then 
\begin{align*}
\BV =T(V)/\bigl(  [x_1x_1x_2x_2]  &- \frac{1}{2}q_{11}q_{12}(q_{12}q_{21}-q_{11})(1-q_{12}q_{21})[x_1x_2]^2, 
 x_1^{3},\  x_2^{3} \bigr)
\end{align*}
with  basis $
 \bigl\{x_2^{r_2}[x_1x_2x_2]^{r_{122} }[x_1x_2]^{r_{12} }[x_1x_1x_2]^{r_{112}} x_1^{r_1}\ |\  0\le r_1,r_{2}< 3,\ 
0 \le r_{112},r_{122}< 2,\ 0\le r_{12}<4\bigr\}
$ 
and $\operatorname{dim}_\k \BV=144.$

\item[\emph{(2)}]  \Dchaintwo{}{$-\zeta ^{-2}$}{$\zeta ^{-1}$}{$-1$},  \Dchaintwo{}{$-\zeta ^2$}{$-\zeta $}{$-1$}. Let $q_{11}=-\zeta^{2}$, $q_{12}q_{21}=-\zeta$,  $q_{22}=-1$, or 
   $q_{11}=-\zeta^{-2}$, $q_{12}q_{21}=\zeta^{-1}$,  $q_{22}=-1$, then 
\begin{align*}
\BV =T(V)/\bigl( [x_1x_1x_2x_1x_2x_1x_2]  , \ x_1^{3},\  x_2^{2}\bigr)
\end{align*}
with  basis 
$ 
\bigl\{x_2^{r_2}[x_1x_2]^{r_{12} }[x_1x_1x_2x_1x_2]^{r_{11212}}[x_1x_1x_2]^{r_{112}} x_1^{r_1}\ |\  0\le r_1,r_{112}< 3,\ 0 \le r_{2},r_{11212}< 2,\,0\le r_{12}<4\bigr\}
$ 
and $\operatorname{dim}_\k\BV=144.$

\item[\emph{(3)}] \Dchaintwo{}{$-\zeta ^3$}{$\zeta $}{$-1$}, \Dchaintwo{}{$-\zeta ^3$}{$-\zeta ^{-1}$}{$-1$}. Let  $q_{11}=-\zeta^{3}$, $q_{12}q_{21}=\zeta$,  $q_{22}=-1$, or   $q_{11}=-\zeta^{3}$, $q_{12}q_{21}=-\zeta^{-1}$,  $q_{22}=-1$, then 
\begin{align*}
\BV =T(V)/\bigl( [x_1x_1x_2x_1x_2]  , \
 x_1^{4},\
 x_2^{2}\bigr)
\end{align*}
with  basis 
$ 
 \bigl\{x_2^{r_2}[x_1x_2]^{r_{12} }[x_1x_1x_2]^{r_{112}}[x_1x_1x_1x_2]^{r_{1112}} x_1^{r_1}\ |\  0\le r_1< 4,\
 0\le r_{12},r_{112}<3,\,  0 \le r_{2},r_{1112}< 2\}
$ 
and $\operatorname{dim}_\k\BV=144.$

\item[\emph{(4)}]  \Dchaintwo{}{$-\zeta ^2$}{$\zeta ^3$}{$-1$}.  Let $q_{11}=-\zeta^{2}$, $q_{12}q_{21}=\zeta^3$,  $q_{22}=-1$, then 
\begin{align*}
\BV =T(V)/\bigl( [x_1x_1x_2x_1x_2x_1x_2]  , \
 x_1^{3},\
[x_1x_2]^{12},\
 x_2^{2}\bigr)
\end{align*}
with basis 
$ 
\bigl\{x_2^{r_2}[x_1x_2]^{r_{12} }[x_1x_1x_2x_1x_2]^{r_{11212}}[x_1x_1x_2]^{r_{112}} x_1^{r_1}\ |\  0\le r_1,r_{112}< 3,\ 0 \le r_{2},r_{11212}< 2,\,0\le r_{12}<12\bigr\}
$ 
 and $\operatorname{dim}_\k \BV=432.$

\item[\emph{(5)}]  \Dchaintwo{}{$-\zeta ^{-1}$}{$-\zeta ^3$}{$-1$}.  Let $q_{11}=-\zeta^{-1}$, $q_{12}q_{21}=-\zeta^3$,  $q_{22}=-1$, then 
\begin{align*}
\BV =T(V)/\bigl( [x_1x_1x_1x_1x_2] ,\
[x_1x_1x_2x_1x_2], \
 x_1^{12},\
 x_2^{2}\bigr)
\end{align*}
with   basis 
$ 
 \bigl\{x_2^{r_2}[x_1x_2]^{r_{12} }[x_1x_1x_2]^{r_{112}}[x_1x_1x_1x_2]^{r_{1112}} x_1^{r_1}\ |\  0\le r_1< 12,\
 0\le r_{12},r_{112}<3,\,  0 \le r_{2},r_{1112}< 2 \bigr\}
$ 
  and $\operatorname{dim}_\k \BV=432.$
\end{enumerate}
\end{prop}

\begin{rem}
The Nichols algebras of Proposition \ref{PropNicholsAlgRows89} have different PBW bases, also if they are in the same Weyl equivalence class. They build up the tree types $T_4,$ $T_5$ and $T_7$ of \cite{HeckBV1}.
\end{rem}

\begin{thm}[Liftings of $\BV$ of rows 8 and 9]\label{ThmLiftRow8} For any lifting $A$ of $\BV$ as in Proposition \ref{PropNicholsAlgRows89}, we have
$$
A\cong (T(V)\# \k[\G])/ I,
$$ 
where $I$ is 
specified as follows: Let $\zeta\in\k^\times$, $\ord\zeta =12$.  
\begin{enumerate}
 \item[\emph{(1)}] 
 \Dchaintwo{}{$-\zeta ^{-2}$}{$-\zeta ^3$}{$-\zeta ^2$}. Let $q_{11}=-\zeta^{-2}$, 
$
q_{12}q_{21}=-\zeta^{3}
$, $q_{22}=-\zeta^{2}$.\\
\emph{(a)} If $q_{12}^3\neq 1$, then $I$ is generated by
\begin{align*}
[x_1x_1x_2x_2] & - \frac{1}{2}
q_{11}q_{12}(q_{12}q_{21}-q_{11})(1-q_{12}q_{21})[x_1x_2]^2, \\
 x_1^{3}&-\mu_1(1-g_1^3),\\
 x_2^{3}&.
\end{align*}
\emph{(b)} If $q_{12}^3= 1$, 
then $I$ is generated by
\begin{align*}
[x_1x_1x_2x_2] & - \frac{1}{2}
q_{11}q_{12}(q_{12}q_{21}-q_{11})(1-q_{12}q_{21})
[x_1x_2]^2, \\
 x_1^{3}&,\\
 x_2^{3}&-\mu_2(1-g_2^3).
\end{align*}
In \emph{(a),(b)}  $\operatorname{dim}_\k A=144\cdot |\G|$ and a  basis is 
\begin{multline*}
 \bigl\{x_2^{r_2}[x_1x_2x_2]^{r_{122} }[x_1x_2]^{r_{12} }[x_1x_1x_2]^{r_{112}} x_1^{r_1}g\ |\  0\le r_1,r_{2}< 3,\\ 
0 \le r_{112},r_{122}< 2,\ 0\le r_{12}<4,\, g\in\G \bigr\}.
\end{multline*}

\item[\emph{(2)}]  \Dchaintwo{}{$-\zeta ^{-2}$}{$\zeta ^{-1}$}{$-1$}, \Dchaintwo{}{$-\zeta ^2$}{$-\zeta $}{$-1$}. Let $q_{11}=-\zeta^{2}$, $q_{12}q_{21}=-\zeta$,  $q_{22}=-1$, or 
   $q_{11}=-\zeta^{-2}$, $q_{12}q_{21}=\zeta^{-1}$,  $q_{22}=-1$.\\
\emph{(a)} If $q_{12}\neq \pm 1$, then $I$ is generated by
\begin{align*}
[x_1x_1x_2x_1x_2x_1x_2]  , \qquad\qquad
 x_1^{3}&-\mu_1(1-g_1^3),\\
 x_2^{2}&.
\end{align*}
\emph{(b)} If $q_{12}= \pm 1$, 
then $I$ is generated by
\begin{align*}
[x_1x_1x_2x_1x_2x_1x_2] & +\mu_2q_{12}(q_{11}q_{12}q_{21}+q_{12}q_{21}-1)[x_1x_1x_2]x_1^2 g_2^2, \\
 x_1^{3}&,\\
 x_2^{2}&-\mu_2(1-g_2^2).
\end{align*}
In \emph{(a),(b)} $\operatorname{dim}_\k A=144\cdot |\G|$ and a basis is 
\begin{multline*}
\bigl\{x_2^{r_2}[x_1x_2]^{r_{12} }[x_1x_1x_2x_1x_2]^{r_{11212}}[x_1x_1x_2]^{r_{112}} x_1^{r_1}g\ |\  0\le r_1,r_{112}< 3,\\ 0 \le r_{2},r_{11212}< 2,\,0\le r_{12}<4,\, g\in\G\bigr\}.
\end{multline*}

\item[\emph{(3)}] \Dchaintwo{}{$-\zeta ^3$}{$\zeta $}{$-1$}, 
\Dchaintwo{}{$-\zeta ^3$}{$-\zeta ^{-1}$}{$-1$}. Let  $q_{11}=-\zeta^{3}$, $q_{12}q_{21}=\zeta$,  $q_{22}=-1$, or   $q_{11}=-\zeta^{3}$, $q_{12}q_{21}=-\zeta^{-1}$,  $q_{22}=-1$ .\\
 \emph{(a)} If $q_{12}\neq \pm 1$, then $I$ is generated by
\begin{align*}
[x_1x_1x_2x_1x_2]  , \qquad\qquad
 x_1^{4}&-\mu_1(1-g_1^3),\\
 x_2^{2}&.
\end{align*}
 \emph{(b)} If $q_{12}= \pm 1$, then $I$ is generated by 
\begin{align*}
[x_1x_1x_2x_1x_2] & -\mu_2 q_{12}(q_{11} + 2q_{12}^2q_{21}^2 - q_{12}q_{21})   x_1^3 g_2^2, \\
 x_1^{4}&,\\
 x_2^{2}&-\mu_2(1-g_2^2).
\end{align*}
In \emph{(a),(b)} $\operatorname{dim}_\k A=144\cdot |\G|$ and a basis is 
\begin{multline*}
 \bigl\{x_2^{r_2}[x_1x_2]^{r_{12} }[x_1x_1x_2]^{r_{112}}[x_1x_1x_1x_2]^{r_{1112}} x_1^{r_1}g\ |\  0\le r_1< 4,\\
 0\le r_{12},r_{112}<3,\,  0 \le r_{2},r_{1112}< 2,\, g\in\G\}.
\end{multline*}

\item[\emph{(4)}]  \Dchaintwo{}{$-\zeta ^2$}{$\zeta ^3$}{$-1$}.  Let $q_{11}=-\zeta^{2}$, $q_{12}q_{21}=\zeta^3$,  $q_{22}=-1$.\\
\emph{(a)} If $q_{12}\neq \pm 1$, then $I$ is generated by
\begin{align*}
[x_1x_1x_2x_1x_2x_1x_2]  , \qquad\qquad
 x_1^{3}&-\mu_1(1-g_1^3),\\
[x_1x_2]^{12}&-\mu_{12}(1-g_{12}^{12}),\\
 x_2^{2}&.
\end{align*}
It is $\operatorname{dim}_\k A=432\cdot |\G|$ and a basis is 
\begin{multline*}
\bigl\{x_2^{r_2}[x_1x_2]^{r_{12} }[x_1x_1x_2x_1x_2]^{r_{11212}}[x_1x_1x_2]^{r_{112}} x_1^{r_1}g\ |\  0\le r_1,r_{112}< 3,\\ 0 \le r_{2},r_{11212}< 2,\,0\le r_{12}<12,\, g\in\G\bigr\}
\end{multline*}
\emph{(b) (incomplete)} $q_{12}= \pm 1$, then $I$ is generated by
\begin{align*}
[x_1x_1x_2x_1x_2x_1x_2] & + q_{12}2\mu_2(q_{12}q_{21}+1)[x_1x_1x_2]x_1^2 g_2^2, \\
 x_1^{3}&,\\
[x_1x_2]^{12}&-\redh{12},\\
 x_2^{2}&-\mu_2(1-g_1^2).
\end{align*}

\item[\emph{(5)}] \Dchaintwo{}{$-\zeta ^{-1}$}{$-\zeta ^3$}{$-1$}. Let $q_{11}=-\zeta^{-1}$, $q_{12}q_{21}=-\zeta^3$,  $q_{22}=-1$.\\
\emph{(a)} If $q_{12}\neq \pm 1$, then $I$ is generated by
\begin{align*}
[x_1x_1x_1x_1x_2]  , \qquad\qquad x_1^{12}&-\mu_1(1-g_1^{12}),\\
[x_1x_1x_2x_1x_2]  , \qquad\qquad \ x_2^{2}&.
\end{align*}
\emph{(b)} If $q_{12}= \pm 1$, 
then $I$ is generated by:
\begin{align*}
[x_1x_1x_1x_1x_2]  , \qquad\qquad\qquad\qquad\quad\ x_1^{12}&-\mu_1(1-g_1^{12}),\\
[x_1x_1x_2x_1x_2]  + 2 \mu_2 q_{12}  x_1^3 g_2^2 , \qquad\qquad
 x_2^{2}&-\mu_2(1-g_2^2).
\end{align*}\\
In  \emph{(a),(b)} $\operatorname{dim}_\k A=432\cdot |\G|$ and a  basis is 
\begin{multline*}
 \bigl\{x_2^{r_2}[x_1x_2]^{r_{12} }[x_1x_1x_2]^{r_{112}}[x_1x_1x_1x_2]^{r_{1112}} x_1^{r_1}g\ |\  0\le r_1< 12,\\
 0\le r_{12},r_{112}<3,\,  0 \le r_{2},r_{1112}< 2,\, g\in\G\bigr\}.
\end{multline*}
\end{enumerate}
\end{thm}
\begin{proof}
We argue exactly as in the proofs of Theorem \ref{ExLiftingsA2} and \ref{ExLiftingsB2}.

\noindent \textbullet \ $(T(V)\# \k[\G])/ I$ is a Hopf algebra, since $I$ is generated by skew-primitive elements:
 The elements $x_i^{N_i}-\mu_i(1-g_i^{N_i})$ and $[x_1x_1x_1x_1x_2]-\lambda_{11112}(1-g_{11112})$ are skew-primitive if $q_{12}q_{21}=q_{11}^{-2}$ by Lemma \ref{LemSkewPrimRootSerre}. For the elements $[x_1x_1x_2x_1x_2]-\red{11212}$ and $[x_1x_1x_2x_1x_2x_1x_2]-\red{1121212}$ we use Lemma \ref{LemSkew11212} and for   $[x_1x_2]^{N_{12}}-\redh{12}$ Lemma \ref{LemSkewZ}(1). Further in (1) $[x_1x_1x_2x_2]  - \frac{1}{2}
q_{11}q_{12}(q_{12}q_{21}-q_{11})(1-q_{12}q_{21})[x_1x_2]^2$ is skew-primitive by a straightforward calculation.  

\noindent \textbullet \ The statement on the basis and dimension of $(T(V)\# \k[\G])/ I$ is proven in \cite{Helbig-PBW}.

\noindent \textbullet \ $(T(V)\# \k[\G])/ I$ is pointed and $\operatorname{gr}((T(V)\# \k[\G])/ I)\cong \BV\#\k[\G]$ by the same arguments as in the proofs of Theorems \ref{ExLiftingsA2} and \ref{ExLiftingsB2}.

\noindent \textbullet \ Also in the same way, the surjective Hopf algebra map 
$
T(V)\#\k[\G]\rightarrow A
$ factorizes to an isomorphism
$$
(T(V)\#\k[\G])/I\stackrel{\sim}{\rightarrow} A
$$
by Lemma \ref{LemLiftofGeneralRel} and \ref{LemQijCombinatorics2}, doing the combinatorics on the braiding matrices.
\end{proof}

\begin{rem}
The Conjecture \ref{ConjSkewPrim} is true in the above cases. Further note that in (1) $r_{1122}\notin \k[\G]$ (as well as $r_{1122}\neq 0$ in $\BV$), in (2b) $r_{1121212}\notin \k[\G]$, in (3b) $r_{11212}\notin \k[\G]$, in (4b) $r_{1121212}\notin \k[\G]$ and in (5b) $r_{11212}\notin \k[\G]$.
\end{rem}

\bibliographystyle{plain}
\bibliography{mybib}
\end{document}